\documentclass[a4paper,reqno,11pt]{amsart}

\usepackage{amsmath}
\usepackage{amssymb}
\usepackage{amsthm}
\usepackage{latexsym}
\usepackage[all]{xypic}
\usepackage{amsfonts}
\usepackage{mathrsfs}
\usepackage{graphicx,color}
\usepackage{xcolor}
\usepackage{tikz}
\usepackage{graphics}
\usetikzlibrary[shapes]
\xyoption{all}
\usepackage[colorlinks=true,linkcolor=blue,citecolor=blue]{hyperref}
\usepackage{anysize,hyperref}
\marginsize{30mm}{30mm}{23mm}{23mm} 

\usepackage{soul}
\setstcolor{red}

\newcommand{\cal}{\mathcal}

\usepackage{enumitem}
\newtheorem{theorem}{Theorem}[section]
\newtheorem{lemma}[theorem]{Lemma}
\newtheorem{corollary}[theorem]{Corollary}
\newtheorem{proposition}[theorem]{Proposition}
\theoremstyle{definition}
\newtheorem{definition}[theorem]{Definition}
\newtheorem{definition-proposition}[theorem]{Definition-Proposition}
\newtheorem{remark}[theorem]{Remark}
\newtheorem{example}[theorem]{Example}

\def\C{\mathcal{C}}
\def\D{\mathcal{D}}

\def\T{\mathcal{T}}

\def\W{\mathscr{W}}
\def\L{\mathscr{L}}
\def\U{\mathcal{U}}
\def\A{\mathcal{A}}
\def\B{\mathcal{B}}
\def\X{\mathcal{X}}
\def\D{\mathcal{D}}
\def\F{\mathcal{F}}
\def\Y{\mathcal{Y}}
\def\V{\mathcal{V}}
\def\Z{\mathcal {Z}}
\def\I{\mathcal {I}}
\def\P{\mathcal {P}}
\def\M{\mathcal{M}}
\def\E{\mathbb{E}}

\def \text{\mbox}

\providecommand{\add}{\mathop{\rm add}\nolimits}%
\providecommand{\Ext}{\mathop{\rm Ext}\nolimits}%
\providecommand{\Hom}{\mathop{\rm Hom}\nolimits}%
\renewcommand{\mod}{\mathop{\rm mod}\nolimits}%
\providecommand{\ind}{\mathop{\rm ind}\nolimits}%
\providecommand{\nc}{\mathop{\rm nc}\nolimits}%
\def\XX{\widetilde{\X}}
\def\YY{\widetilde{\Y}}

\usetikzlibrary{arrows,decorations.pathmorphing,decorations.pathreplacing,positioning,shapes.geometric,shapes.misc,decorations.markings,decorations.fractals,calc,patterns}
\begin{document}

\title{Mutation of $n$-cotorsion pairs in extriangulated categories}

\author[Chang]{Huimin Chang}
\address{Department of Applied Mathematics,
The Open University of China,
100039 Beijing,
P. R. China
}
\email{changhm@ouchn.edu.cn}

\author[Liu]{Yu Liu}
\address{ School of Mathematics and Statistics, Shaanxi Normal University, 710062 Xi'an, Shannxi, P. R. China}
\email{recursive08@hotmail.com}

\author[Zhou]{Panyue Zhou$^\ast$}
\address{School of Mathematics and Statistics, Changsha University of Science and Technology, 410114 Changsha, Hunan,  P. R. China}
\email{panyuezhou@163.com}

\thanks{$^\ast$Corresponding author.}

\begin{abstract}
In this article, we introduce the notion of $n$-cotorsion pairs in extriangulated categories, which extends both the cotorsion pairs established by Nakaoka and Palu and the $n$-cotorsion pairs in triangulated categories developed by Chang and Zhou. We further prove that any mutation of an $n$-cotorsion pair remains an $n$-cotorsion pair.
As applications, we provide a geometric characterization of $n$-cotorsion pairs in $n$-cluster categories of type $A_{\infty}$,
and we realize mutations of
$n$-cotorsion pairs geometrically via rotations of certain configurations of
$n$-admissible arcs.
\end{abstract}

\dedicatory{Dedicated to Professor Bin Zhu on the occasion of his 60th birthday}



\thanks{Huimin Chang is supported by the National Natural Science Foundation of China (Grant No. 12301047). Yu Liu is supported by the National Natural Science Foundation of China (Grant No. 12171397). Panyue Zhou is supported by the National Natural Science Foundation of China (Grant No. 12371034) and by the Scientific Research Fund of Hunan Provincial Education Department (Grant No. 24A0221). }

\maketitle

\section{Introduction}
Extriangulated categories were introduced by Nakaoka and Palu in their foundational work \cite{NP} as a common generalization of exact categories and triangulated categories. This framework encompasses both exact categories, including abelian ones, and extension-closed subcategories of triangulated categories as particular cases. In addition, a number of examples have been constructed that fall outside the realm of either exact or triangulated categories, as shown in \cite{NP,ZZ1,ZZ20,HZZ,FHZZ}, thereby demonstrating the flexibility and breadth of the extriangulated setting.

Let $n$ be a positive integer. Motivated by structural properties of Gorenstein projective and Gorenstein injective modules over Iwanaga-Gorenstein rings, Huerta, Mendoza, and P\'{e}rez~\cite[Definition~2.2]{HMP} formulated the concept of $n$-cotorsion pairs (as well as their left and right variants) in abelian categories. Notably, the case $n = 1$ recovers the classical notion of complete cotorsion pairs.
Building on this idea, He and Zhou~\cite{HZ} extended the definition to the framework of extriangulated categories. In particular, when $n = 1$, their definition coincides with the cotorsion pairs introduced by Nakaoka and Palu in~\cite[Definition~4.1]{NP}.
More recently, Chang and Zhou~\cite{CZ} proposed a notion of $n$-cotorsion pairs in triangulated categories, further generalizing the classical theory. It is important to note that their construction differs from the one in~\cite{HZ}, even when restricted to triangulated settings, thereby giving rise to distinct notions under the same terminology.
For more details, see \cite[Remark 3.4 and Example 3.5]{CZ}.

Iyama and Yoshino~\cite{IY} introduced the notion of mutation in triangulated categories, and subsequently, Zhou and Zhu~\cite{ZZ} proved that the mutation of a torsion pair remains a torsion pair. In a further development, Zhou and Zhu~\cite{ZZ1} formulated the concept of mutation pairs of subcategories in extriangulated categories, extending Iyama-Yoshino's definition~\cite[Definition~2.5]{IY} from the triangulated setting.

The concept of $n$-cluster categories originates from the study of cluster algebras and higher-dimensional homological algebra. Cluster categories were initially introduced as a categorical model to realize the combinatorial structures of cluster algebras, which were developed by Fomin and Zelevinsky \cite{FZ}. In the classical case where $n=1$, cluster categories provided a triangulated framework to understand tilting theory and the process of mutation among tilting objects.
To generalize these ideas to higher dimensions, Iyama \cite{I} developed higher Auslander--Reiten theory and introduced the notion of $n$-cluster tilting subcategories. Building upon this, Keller and others \cite{K,BM,BM1,Ja}constructed $n$-cluster categories as orbit categories of bounded derived categories, which encode the properties of $n$-cluster tilting theory and provide a natural setting for higher-dimensional analogues of classical tilting.

The study of $n$-cluster categories is significant for several reasons. First, they provide a rich framework in which to explore the mutation and classification of $n$-cluster tilting objects. Second, they serve as a bridge between combinatorial and categorical approaches to higher cluster algebras. Third, they have many nice properties, for instance Hom-finite, Krull–Schmidt, and $(n+1)$-Calabi--Yau, a property that connects them to important topics in geometry and physics, including mirror symmetry and string theory.
Moreover, the theory of $n$-cluster categories has led to the development of new algebraic structures such as $n$-abelian categories \cite{G} and higher cotorsion pairs \cite{HMP,HZ,CZ}, enriching the landscape of modern representation theory. These structures also provide a unifying language that links derived categories, triangulated and extriangulated categories, and higher homological techniques. As such, $n$-cluster categories remain a central and dynamic area of research in higher representation theory.

Classifying cotorsion pairs is an essential step toward understanding the internal structure of categories. In triangulated settings, cotorsion pairs are closely connected to the theories of $t$-structures and co-$t$-structures, which are fundamental tools in the study of derived categories and algebraic geometry. In extriangulated categories, which generalize both exact and triangulated categories, cotorsion pairs offer a natural method for analyzing extension theory in settings that do not conform strictly to either traditional framework.
Furthermore, cotorsion pair classification has deep connections with silting theory, support $\tau$-tilting theory, and the theory of cluster tilting subcategories. These relationships have practical implications for constructing and identifying key categorical objects such as rigid and cluster-tilting objects, as well as for understanding approximation theory.
In the context of higher homological algebra, classifying $n$-cotorsion pairs provides insight into higher extension vanishing phenomena and their relation to $n$-cluster tilting theory. This direction of research continues to expand the understanding of homological dimensions and has become an important theme in modern representation theory and categorical algebra.

In this paper, we propose a weaker version of $n$-cotorsion pairs in extriangulated categories, compared to the definition given by He and Zhou~\cite{HZ}. Our formulation not only generalizes the cotorsion pairs introduced by Nakaoka and Palu but also encompasses the $n$-cotorsion pairs in triangulated categories developed by Chang and Zhou. We further show that any mutation of an $n$-cotorsion pair yields another $n$-cotorsion pair.
As applications, we provide a geometric characterization of $n$-cotorsion pairs in $n$-cluster categories of type $A_{\infty}$, and offer a geometric interpretation of their mutations via rotations of certain configurations of $n$-admissible arcs.
We note that an $n$-cluster category of type $A_{\infty}$ contains infinitely many indecomposable objects. Moreover, its Auslander-Reiten quiver consists of $n$ components, making the structure significantly more intricate than in the finite case.

This paper is organized as follows. In Section~2, we recall basic notions of cotorsion pairs and mutation pairs in extriangulated categories. In Section~3, we introduce the definition of $n$-cotorsion pairs, review the construction of subfactor triangulated categories from~\cite{ZZ1}, investigate its compatibility with our setting, and establish the main result. In Section~4, we give a geometric characterization of $n$-cotorsion pairs in $n$-cluster categories of type $A_{\infty}$, and a geometric interpretation of their mutations via rotations of certain configurations of $n$-admissible arcs.

\subsection*{Conventions}
Throughout this paper, we assume that $\mathcal{C}$ is an extriangulated category with enough projective and enough injective objects. An object $P \in \mathcal{C}$ is said to be \emph{projective} if $\mathbb{E}(P, \mathcal{C}) = 0$, and we denote by $\mathcal{P} \subseteq \mathcal{C}$ the full subcategory consisting of projective objects. The category $\mathcal{C}$ is said to have \emph{enough projectives} if every object $C \in \mathcal{C}$ admits a deflation $P \twoheadrightarrow C$ for some $P \in \mathcal{P}$.
Dually, an object $I \in \mathcal{C}$ is called \emph{injective} if $\mathbb{E}(\mathcal{C}, I) = 0$, and we denote by $\mathcal{I} \subseteq \mathcal{C}$ the subcategory of injective objects. The category $\mathcal{C}$ is said to have \emph{enough injectives} if every object $C \in \mathcal{C}$ admits an inflation $C \rightarrowtail I$ for some $I \in \mathcal{I}$.

\section{Preliminaries}
This section provides background material on extriangulated categories, cotorsion pairs, and mutations of subcategories, which will be used throughout the paper.

\subsection{Extriangulated categories}
We begin by summarizing the key terminology and properties of extriangulated categories relevant to our context. For full definitions and foundational results, we refer the reader to~\cite{NP}.

Let $\mathcal{C}$ be an additive category equipped with an additive bifunctor
\[
\mathbb{E} \colon \mathcal{C}^{\mathrm{op}} \times \mathcal{C} \longrightarrow \mathrm{Ab},
\]
where $\mathrm{Ab}$ denotes the category of abelian groups. For any pair of objects $A, C \in \mathcal{C}$, an element $\delta \in \mathbb{E}(C, A)$ is referred to as an \emph{$\mathbb{E}$-extension}.

Let $\mathfrak{s}$ be a correspondence that assigns to each $\mathbb{E}$-extension $\delta \in \mathbb{E}(C, A)$ an equivalence class
\[
\mathfrak{s}(\delta) = \xymatrix@C=0.8cm{[A \ar[r]^x & B \ar[r]^y & C]}.
\]
Such a correspondence $\mathfrak{s}$ is called a \emph{realization} of $\mathbb{E}$ if it makes the diagrams in~\cite[Definition~2.9]{NP} commute.

A triplet $(\mathcal{C}, \mathbb{E}, \mathfrak{s})$ is called an \emph{extriangulated category} if the following conditions are satisfied:
\begin{enumerate}
  \item The functor $\mathbb{E} \colon \mathcal{C}^{\mathrm{op}} \times \mathcal{C} \to \mathrm{Ab}$ is additive bifunctor.
  \item The correspondence $\mathfrak{s}$ is an additive realization of $\mathbb{E}$.
  \item The pair $(\mathbb{E}, \mathfrak{s})$ satisfies the compatibility axioms given in~\cite[Definition~2.12]{NP}.
\end{enumerate}

For notational convenience, we adopt the convention that $\mathcal{C}$ denotes the extriangulated category $(\mathcal{C}, \mathbb{E}, \mathfrak{s})$ whenever no confusion arises.
\vspace{1mm}

We will use the following terminology in the sequel.

\begin{definition}\cite[Definitions 2.15 and 2.19]{NP}
Let $\mathcal{C}$ be an extriangulated category.
\begin{enumerate}
  \item[{\rm (1)}] A sequence $A \xrightarrow{x} B \xrightarrow{y} C$ is called a \emph{conflation} if it realizes an extension $\delta \in \mathbb{E}(C, A)$. In this case, $x$ is referred to as an \emph{inflation}, and $y$ as a \emph{deflation}.

  \item[{\rm (2)}] If a conflation $A \xrightarrow{x} B \xrightarrow{y} C$ realizes the extension $\delta \in \mathbb{E}(C, A)$, the pair $(A \xrightarrow{x} B \xrightarrow{y} C, \delta)$ is called an \emph{$\mathbb{E}$-triangle}, which we write as
  \[
  A \overset{x}{\longrightarrow} B \overset{y}{\longrightarrow} C \overset{\delta}{\dashrightarrow}.
  \]
  When the extension class $\delta$ is not explicitly used in the argument, we omit it for brevity.

  \item[{\rm (3)}] Let $A \overset{x}{\longrightarrow} B \overset{y}{\longrightarrow} C \overset{\delta}{\dashrightarrow}$ and $A' \overset{x'}{\longrightarrow} B' \overset{y'}{\longrightarrow} C' \overset{\delta'}{\dashrightarrow}$ be two $\mathbb{E}$-triangles. A triplet of morphisms $(a, b, c)$ is called a \emph{morphism of $\mathbb{E}$-triangles} if it realizes a morphism $(a, c)\colon \delta \to \delta'$. In this case, we write the following commutative diagram:
  \[
  \xymatrix{
  A \ar[r]^x \ar[d]^a & B \ar[r]^y \ar[d]^b & C \ar@{-->}[r]^{\delta} \ar[d]^c & \\
  A' \ar[r]^{x'} & B' \ar[r]^{y'} & C' \ar@{-->}[r]^{\delta'} &. }
  \]
\end{enumerate}
\end{definition}

Let $\mathcal{C}$ be an extriangulated category with enough projectives and enough injectives, and let $\mathcal{X}$ be a subcategory of $\mathcal{C}$.
We define $\Omega \mathcal{X} := \mathrm{CoCone}(\mathcal{P}, \mathcal{X})$, that is, the subcategory of $\mathcal{C}$ consisting of objects $\Omega X$ that fit into an $\mathbb{E}$-triangle
\[
\Omega X \overset{a}{\longrightarrow} P \overset{b}{\longrightarrow} X \overset{}{\dashrightarrow}
\]
with $P \in \mathcal{P}$ and $X \in \mathcal{X}$. We refer to $\Omega$ as the \emph{syzygy} of $\mathcal{X}$.

Dually, the \emph{cosyzygy} of $\mathcal{X}$ is defined by $\Sigma \mathcal{X} := \mathrm{Cone}(\mathcal{X}, \mathcal{I})$, which consists of objects $\Sigma X$ that appear in an $\mathbb{E}$-triangle
\[
X \overset{c}{\longrightarrow} I \overset{d}{\longrightarrow} \Sigma X \overset{}{\dashrightarrow}
\]
with $I \in \mathcal{I}$ and $X \in \mathcal{X}$. For further details, see~\cite[Definition~4.2 and Proposition~4.3]{LN}.

For a subcategory $\mathcal{X} \subseteq \mathcal{C}$, we define $\Omega^0 \mathcal{X} := \mathcal{X}$, and inductively set
\[
\Omega^k \mathcal{X} := \Omega(\Omega^{k-1} \mathcal{X}) = \mathrm{CoCone}(\mathcal{P}, \Omega^{k-1} \mathcal{X})
\quad \text{for all}~ k > 0.
\]
We call $\Omega^k \mathcal{X}$ the \emph{$k$-th syzygy} of $\mathcal{X}$. The \emph{$k$-th cosyzygy} $\Sigma^k \mathcal{X}$ is defined dually by setting $\Sigma^0 \mathcal{X} := \mathcal{X}$ and
\[
\Sigma^k \mathcal{X} := \mathrm{Cone}(\Sigma^{k-1} \mathcal{X}, \mathcal{I})
\quad\text{for all}~ k > 0.
\]

Liu and Nakaoka~\cite{LN} introduced higher extension groups in extriangulated categories with enough projectives and enough injectives by defining
\[
\mathbb{E}(A, \Sigma^k B) \cong \mathbb{E}(\Omega^k A, B) \quad \text{for all } k \geq 0.
\]
For convenience, we denote this group by
\[
\mathbb{E}^{k+1}(A, B) := \mathbb{E}(A, \Sigma^k B) \cong \mathbb{E}(\Omega^k A, B).
\]
They proved the following result.

\begin{lemma}\label{2}
Let $\mathcal{C}$ be an extriangulated category with enough projectives and injectives. Assume that
\[
\xymatrix{
A \ar[r]^{f} & B \ar[r]^{g} & C \ar@{-->}[r]^{\delta} &
}
\]
is an $\mathbb{E}$-triangle in $\mathcal{C}$. Then for any object $X \in \mathcal{C}$ and any integer $k \geq 1$, the following long exact sequences hold:
\[
\cdots \rightarrow \mathbb{E}^k(X, A) \rightarrow \mathbb{E}^k(X, B) \rightarrow \mathbb{E}^k(X, C)
\rightarrow \mathbb{E}^{k+1}(X, A) \rightarrow \mathbb{E}^{k+1}(X, B) \rightarrow \cdots;
\]
\[
\cdots \rightarrow \mathbb{E}^k(C, X) \rightarrow \mathbb{E}^k(B, X) \rightarrow \mathbb{E}^k(A, X)
\rightarrow \mathbb{E}^{k+1}(C, X) \rightarrow \mathbb{E}^{k+1}(B, X) \rightarrow \cdots.
\]
\end{lemma}

As a higher-dimensional analogue of cluster tilting subcategories in extriangulated categories~\cite[Definition~4.1]{CZZ}, Liu and Nakaoka~\cite[Definition~5.3]{LN} introduced the notion of \emph{$n$-cluster tilting subcategories} in the same context. This generalizes Iyama's definition~\cite[Definition~1.1]{I} in the abelian case.

\begin{definition}\cite[Definition 5.3]{LN}
Let $\mathcal{C}$ be an extriangulated category with enough projectives and injectives. A subcategory $\mathcal{X} \subseteq \mathcal{C}$ is called \emph{$n$-cluster tilting} if the following conditions hold:
\begin{enumerate}
\item[{\rm (1)}] $\mathcal{X}$ is both contravariantly and covariantly finite in $\mathcal{C}$;
\item[{\rm (2)}] $M \in \mathcal{X}$ if and only if $\mathbb{E}^k(\mathcal{X}, M) = 0$ for all $1 \leq k \leq n-1$;
\item[{\rm (3)}] $M \in \mathcal{X}$ if and only if $\mathbb{E}^k(M, \mathcal{X}) = 0$ for all $1 \leq k \leq n-1$.
\end{enumerate}
\end{definition}

Let $\mathcal{X}, \mathcal{Y} \subseteq \mathcal{C}$ be subcategories and $k \geq 1$ an integer. We write $\mathbb{E}^k(\mathcal{X}, \mathcal{Y}) = 0$ to mean that $\mathbb{E}^k(X, Y) = 0$ for all $X \in \mathcal{X}$ and $Y \in \mathcal{Y}$. If either $\mathcal{X} = \{M\}$ or $\mathcal{Y} = \{N\}$ is a singleton, we write $\mathbb{E}^k(M, \mathcal{Y}) = 0$ or $\mathbb{E}^k(\mathcal{X}, N) = 0$ accordingly.

We define the \emph{right $k$-th orthogonal} of $\mathcal{X}$ as
\[
\mathcal{X}^{\perp_k} := \{ N \in \mathcal{C} \mid \mathbb{E}^k(\mathcal{X}, N) = 0 \}.
\]
Dually, the \emph{left $k$-th orthogonal} of $\mathcal{Y}$ is defined by
\[
{}^{\perp_k} \mathcal{Y} := \{ M \in \mathcal{C} \mid \mathbb{E}^k(M, \mathcal{Y}) = 0 \}.
\]

It is easy to verify that a subcategory $\mathcal{X}$ is $n$-cluster tilting if and only if it is contravariantly and covariantly finite in $\mathcal{C}$ and satisfies
\[
\mathcal{X} = \bigcap_{k=1}^{n-1} {}^{\perp_k} \mathcal{X} = \bigcap_{k=1}^{n-1} \mathcal{X}^{\perp_k}.
\]

By the definition of higher extension groups, one readily sees that $\mathbb{E}^k(P, \mathcal{C}) = 0$ for all $k \geq 1$ if $P$ is projective, and $\mathbb{E}^k(\mathcal{C}, I) = 0$ for all $k \geq 1$ if $I$ is injective. It follows that if $\mathcal{X}$ is $n$-cluster tilting, then $\mathcal{P} \subseteq \mathcal{X}$ and $\mathcal{I} \subseteq \mathcal{X}$.

\subsection{Cotorsion pairs in extriangulated categories}
In this subsection, we recall the notion of cotorsion pairs in extriangulated categories.

\begin{definition}\cite[Definition 4.1]{NP}
Let $\mathcal{U}$ and $\mathcal{V}$ be full additive subcategories of $\mathcal{C}$, closed under direct summands and isomorphisms. The pair $(\mathcal{U}, \mathcal{V})$ is called a \emph{cotorsion pair} if the following conditions are satisfied:
\begin{itemize}
  \item[(a)] $\mathbb{E}(\mathcal{U}, \mathcal{V}) = 0$;
  \item[(b)] The pair $(\mathcal{U}, \mathcal{V})$ is \emph{complete} in the sense of~\cite{H}, that is, for every object $C \in \mathcal{C}$, there exist two $\E$-triangles
  \[
  V_C \to U_C \to C\dashrightarrow,\quad C\to V^C \to U^C\dashrightarrow
  \]
  with $U_C, U^C \in \mathcal{U}$ and $V_C, V^C \in \mathcal{V}$.
\end{itemize}
\end{definition}

Since $\mathcal{C}$ has enough projective and injective objects, cotorsion pairs in $\mathcal{C}$ admit an equivalent characterization.

\begin{proposition}\label{1}
Suppose that $\C$ is an extriangulated category with enough projectives and enough injectives, $\U$ and $\V$ are full additive subcategories of $\C$ which are closed under direct summands and isomorphisms. Then $(\U, \V)$ is a cotorsion pair if and only if the following conditions are satisfied.
\begin{itemize}
  \item [\rm (1)]~ $\U={^{\perp_1}\V}$.
  \item [\rm (2)]~ $\V=\U^{\perp_1}$.
  \item [\rm (3)]~ $\U$ is contravariantly finite and $\V$ is covariantly finite in $\C$.
\end{itemize}
\end{proposition}
\begin{proof}
Suppose that $(\U, \V)$ is a cotorsion pair. Then it is easy to check conditions (1), (2) and (3) are satisfied. Conversely, suppose conditions (1), (2) and (3) are satisfied, then $\E(\U, \V) =0$. Moreover, for any project object $P\in\P$, since $\E(P,\C)=0$, we have $\E(P,\V)=0$, then $\P\subseteq{^{\perp_1}\V}=\U$. Similarly, we can prove $\I\subseteq\V$. By \cite[Remark 3.20]{ZZ1}, $\U$ is strongly contravariantly finite and $\V$ is strongly covariantly finite in $\C$. That means for any object $C\in\C$, there exist two $\mathbb{E}$-triangles
  $$X_C\rightarrow U_C\overset{f}\rightarrow C\dashrightarrow,$$
where $f$ is a right $\U$-approximation of $C$, and
$$C\overset{g}\rightarrow V^C\rightarrow Y^C\dashrightarrow,$$
where $g$ is a left $\V$-approximation of $C$. This implies $\E(\U,X_C)=0$ and $\E(Y^C,V)=0$, so $X_C\in\U^{\perp_1}=\V$ and $Y^C\in{^{\perp_1}\V}=\U$. Thus $(\U, \V)$ is a cotorsion pair.
\end{proof}

\begin{definition}
Let $(\U, \V)$ be a cotorsion pair of $\C$. We call $\I(\U)=\U\cap\V$ the $core$ of the cotorsion pair $(\U, \V)$.
\end{definition}
\begin{remark}
Recall that for a subcategory $\X$ of $\C$, $\X$ is called an $n$-rigid subcategory if $\E^k(\X,\X)=0$ for $1\leq k\leq n-1$. For a cotorsion pair $(\U, \V)$, the core $\I(\U)$ is always $2$-rigid. $(\U, \U)$ is a cotorsion pair if and only if $\U$ is $2$-cluster tilting.
\end{remark}

\subsection{Mutation pairs of extriangulated categories}
In this subsection, we review from \cite{ZZ1} the notion of mutation pairs of subcategories in an extriangulated category.

\begin{definition}\cite[Definition 3.2]{ZZ1}\label{3}
Let $\X$, $\U$ and $\V$ be subcategories of an extriangulated category $\C$, and $\X\subseteq\U$, $\X\subseteq\V$. The pair $(\U, \V)$ is called an $\X$-\emph{mutation pair} if it satisfies
\begin{itemize}
  \item [(1)] For any $U\in\U$, there exists an $\E$-triangle
  $$\xymatrix{U\ar[r]^{f}& X\ar[r]^{g}&V\ar@{-->}[r]^{\delta}&}$$
  where $V\in\V$, and $f$ is a left $\X$-approximation of $U$, $g$ is a right $\X$-approximation of $V$.
  \item [(2)] For any $V\in\V$, there exists an $\E$-triangle
  $$\xymatrix{U\ar[r]^{f}& X\ar[r]^{g}&V\ar@{-->}[r]^{\delta}&}$$
  where $U\in\U$, and $f$ is a left $\X$-approximation of $U$, $g$ is a right $\X$-approximation of $V$.
\end{itemize}
\end{definition}

\vspace{2mm}

\section{Mutation of $n$-cotorsion pairs in extriangulated categories}
In this section, we fix an integer $n\geq 1$.

\subsection{$n$-cotorsion pairs in extriangulated categories}

We define $n$-cotorsion pairs in  an extriangulated category, which is a generalization of the cotorsion pairs in the sense of Nakaoka and Palu \cite{NP}.
\begin{definition}\label{aaa}
Let $\C$ be an extriangulated category and $\U,\V$ be two subcategories of $\C$ which are closed
under direct summands and isomorphisms. The pair $(\U,\V)$ is called an $n$-\emph{cotorsion pair} if the following conditions hold.
\begin{enumerate}
  \item [(1)] $\U=\bigcap\limits_{k=1}^{n}{^{\perp_k}}\V$~~ for all $1\leq k\leq n$;
  \item [(2)] $\V=\bigcap\limits_{k=1}^{n}\U^{\perp_k}$~~ for all $1\leq k\leq n$;
   \item [(3)] $\U$ is contravariantly finite and $\V$ is covariantly finite.
\end{enumerate}
\end{definition}

\begin{remark}
By Proposition \ref{1}, the concept of $1$-cotorsion pair is compatible with the classical definition of a cotorsion pair in the sense of Nakaoka and Palu \cite{NP}. It is easy to check that the pair $(\U,\U)$ is an $n$-cotorsion pair in $\C$ if and only if $\U$ is $(n+1)$-cluster tilting.
\end{remark}
We note that the concept of $n$-cotorsion pair defined in this paper is a  weaker version than that in \cite{HZ} by He and Zhou. When $\C$ is an abelian category, an $n$-cotorsion pair defined in \cite{HMP} can imply an $n$-cotorsion pair in the sense of Definition \ref{aaa}. When $\C$ is a triangulated category, an $n$-cotorsion pair in the sense of Definition \ref{aaa} is compatible with Chang and Zhou defined in \cite{CZ}. There are a list of detailed examples of $n$-cotorsion pairs.

\begin{example}
\begin{enumerate}
 \item[(1)] Let $\C$ be an abelian category and $\A,\B$ be two classes of objects of $\C$ which are closed
under direct summands and isomorphisms. In \cite{HMP}, they defined the concept of left $n$-cotorsion pair $(\A,\B)$ if $\Ext^i(\A,\B)=0$ for every $1\leq i\leq n$ and for every object $C\in\C$, there exists an exact sequence
$$0\rightarrow B_{n-1}\rightarrow B_{n-2}\rightarrow\cdots\rightarrow B_1\rightarrow B_0\rightarrow A\rightarrow C\rightarrow 0$$
with $A\in\A$ and $B_k\in\B$ for every $0\leq k\leq n-1$. Dually, a right $n$-cotorsion pair was defined. The pair $(\A,\B)$ is called an $n$-cotorsion pair in the sense of \cite{HMP} if it is both a left $n$-cotorsion pair and a right $n$-cotorsion pair.

By Theorem 2.7, Proposition 3.1 and their dual in \cite{HMP}, if  $(\A,\B)$ is an $n$-cotorsion pair in the sense of \cite{HMP}, then $(\A,\B)$ is an $n$-cotorsion pair in the sense of Definition \ref{aaa}.

By Proposition 5.8 in \cite{HMP}, suppose $R$ is a Von Neumann regular ring, we denote the classes of Ding projective $R$-modules by $\D\P(R)$, and denote the class of flat $R$-modules by $\F(R)$. Then $(\D\P(R),\F(R))$ is an $n$-cotorsion pair in Mod($R$).
  \item [(2)] Let $\C$ be a triangulated category. Then Definition \ref{aaa} is compatible with Definition 3.1 in \cite{CZ}. By Example 3.5 in \cite{CZ}, let $\C_{A_3}^2$ be the 2-cluster category of type $A_3$. The Auslander-Reiten quiver  of $\C_{A_3}^2$ is shown in Figure \ref{Auslander-Reiten quiver }.
      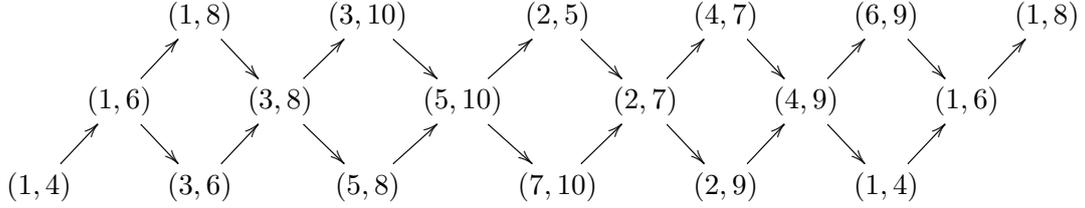
\begin{figure}[h]
\centering

$$
{
\xymatrix@-7mm@C-0.17cm{
     &&(1,8)  \ar[rdd] & & (3,10) \ar[rdd] && (2,5) \ar[rdd]&& (4,7) \ar[rdd]&& (6,9)\ar[rdd] && (1,8) \\
 \\
 & (1,6)  \ar[rdd] \ar[ruu] & & (3,8) \ar[rdd]\ar[ruu] & & (5,10) \ar[rdd] \ar[ruu] & & (2,7) \ar[ruu]\ar[rdd]&& (4,9) \ar[ruu]\ar[rdd]&&(1,6) \ar[ruu]\\
 \\
  (1,4) \ar[ruu]&& (3,6) \ar[ruu] && (5,8) \ar[ruu]&&(7,10) \ar[ruu]&& (2,9) \ar[ruu]&& (1,4)\ar[ruu] && \\
}
}
$$
\caption{The Auslander-Reiten quiver  of $\C_{A_{3}}^{2}$}
\label{Auslander-Reiten quiver }
\end{figure}
Let $\X$ be a subcategory generated by $(1,4),(1,8)$ and $(3,8)$, and $\Y$ be a subcategory generated by $(1,8),(5,8)$, and $(4,7)$. Then $(\X,\Y)$ is a $2$-cotorsion pair.
\end{enumerate}
\end{example}

\begin{example}
{\upshape 	 Let $A=kQ/I$ be an algebra given by the quiver
	 \begin{align}
	 	\begin{minipage}{0.6\hsize}
	   \xymatrix{Q\colon ~ \begin{smallmatrix}1\end{smallmatrix}\ar[r]^{\alpha_1}&\begin{smallmatrix}2\end{smallmatrix}\ar[r]^{\alpha_2}
	 			&\begin{smallmatrix}3\end{smallmatrix}\ar[r]^{\alpha_3}&\begin{smallmatrix}4\end{smallmatrix}\ar[r]^{\alpha_4}&\begin{smallmatrix}5\end{smallmatrix}\ar[r]^{\alpha_5}&\begin{smallmatrix}6\end{smallmatrix}\ar[r]^{\alpha_6}&\begin{smallmatrix}7\end{smallmatrix}\ar[r]^{\alpha_7}&\begin{smallmatrix}8\end{smallmatrix}\ar[r]^{\alpha_8}&\begin{smallmatrix}9\end{smallmatrix}}\notag
	 	\end{minipage}
	 \end{align}
and $I=\langle  \alpha_{i}\alpha_{i+1}\alpha_{i+2}\alpha_{i+3}  \rangle,~i=1,2,3,4,5.$	By \cite[Proposition 3.30]{NP}, the ideal quotient category $\C:=(\mod A)/(\mathcal P\cap \mathcal I)$ (where $\mathcal P$ is the subcategory of projectives and $\mathcal I$ is the subcategory of injectives) is an extriangulated category. In this example, we denote by ``$\bullet$" in the Auslander-Reiten quiver of the $\C$ objects belong to a subcategory and
by  ``$\circ $" the objects do not.}
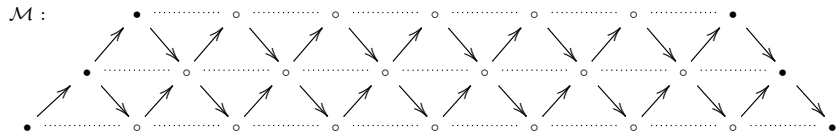
\begin{align}
	\tiny{\xymatrix @R=4mm @C2mm{
			\mathcal M:&&{\begin{smallmatrix}\bullet\end{smallmatrix}}\ar[rd]&&
			{\begin{smallmatrix}\circ\end{smallmatrix}}\ar[rd]\ar@{.}[ll]&&
			{\begin{smallmatrix}\circ\end{smallmatrix}}\ar[rd]\ar@{.}[ll]&&
			{\begin{smallmatrix}\circ\end{smallmatrix}}\ar[rd]\ar@{.}[ll]&&
			{\begin{smallmatrix}\circ\end{smallmatrix}}\ar[rd]\ar@{.}[ll]&&
			{\begin{smallmatrix}\circ\end{smallmatrix}}\ar[rd]\ar@{.}[ll]&&
			{\begin{smallmatrix}\bullet\end{smallmatrix}}\ar[rd]\ar@{.}[ll]&&\\
			&{\begin{smallmatrix}\bullet\end{smallmatrix}}\ar[ru]\ar[rd]&&
			{\begin{smallmatrix}\circ\end{smallmatrix}}\ar[ru]\ar[rd]\ar@{.}[ll]&&
			{\begin{smallmatrix}\circ\end{smallmatrix}}\ar[ru]\ar[rd]\ar@{.}[ll]&&
			{\begin{smallmatrix}\circ\end{smallmatrix}}\ar[ru]\ar[rd]\ar@{.}[ll]&&
			{\begin{smallmatrix}\circ\end{smallmatrix}}\ar[ru]\ar[rd]\ar@{.}[ll]&&
			{\begin{smallmatrix}\circ\end{smallmatrix}}\ar[ru]\ar[rd]\ar@{.}[ll]&&
			{\begin{smallmatrix}\circ\end{smallmatrix}}\ar[ru]\ar[rd]\ar@{.}[ll]&&
			{\begin{smallmatrix}\bullet\end{smallmatrix}}\ar[rd]\ar@{.}[ll]\\
			{\begin{smallmatrix}\bullet\end{smallmatrix}}\ar[ru]&&\textnormal{$\begin{smallmatrix}\circ\end{smallmatrix}$}\ar[ru]\ar@{.}[ll]&&
			{\begin{smallmatrix}\circ\end{smallmatrix}}\ar[ru]\ar@{.}[ll]&&\textnormal{$\begin{smallmatrix}\circ\end{smallmatrix}$}\ar[ru]\ar@{.}[ll]&&
			{\begin{smallmatrix}\circ\end{smallmatrix}}\ar[ru]\ar@{.}[ll]&&
			{\begin{smallmatrix}\circ\end{smallmatrix}}\ar[ru]\ar@{.}[ll]&&
			{\begin{smallmatrix}\circ\end{smallmatrix}}\ar[ru]\ar@{.}[ll]&&
			{\begin{smallmatrix}\circ\end{smallmatrix}}\ar[ru]\ar@{.}[ll]&&
			{\begin{smallmatrix}\bullet\end{smallmatrix}}\ar@{.}[ll]\\
	}}\notag
\end{align}
$\mathcal M$ is a $4$-cluster tilting subcategory of $\C$, hence $(\mathcal M,\mathcal M)$ is a $3$-cotorsion pair. We give two examples of $3$-cotorsion pairs $(\U_1,\V_1)$ and $(\U_2,V_2)$ in $\C$, where
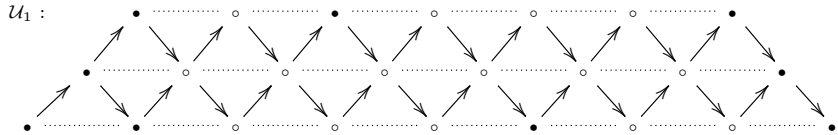
\begin{align}
	\tiny{\xymatrix @R=4mm @C2mm{
			\mathcal U_1:&&{\begin{smallmatrix}\bullet\end{smallmatrix}}\ar[rd]&&
			{\begin{smallmatrix}\circ\end{smallmatrix}}\ar[rd]\ar@{.}[ll]&&
			{\begin{smallmatrix}\bullet\end{smallmatrix}}\ar[rd]\ar@{.}[ll]&&
			{\begin{smallmatrix}\circ\end{smallmatrix}}\ar[rd]\ar@{.}[ll]&&
			{\begin{smallmatrix}\circ\end{smallmatrix}}\ar[rd]\ar@{.}[ll]&&
			{\begin{smallmatrix}\circ\end{smallmatrix}}\ar[rd]\ar@{.}[ll]&&
			{\begin{smallmatrix}\bullet\end{smallmatrix}}\ar[rd]\ar@{.}[ll]&&\\
			&{\begin{smallmatrix}\bullet\end{smallmatrix}}\ar[ru]\ar[rd]&&
			{\begin{smallmatrix}\circ\end{smallmatrix}}\ar[ru]\ar[rd]\ar@{.}[ll]&&
			{\begin{smallmatrix}\circ\end{smallmatrix}}\ar[ru]\ar[rd]\ar@{.}[ll]&&
			{\begin{smallmatrix}\circ\end{smallmatrix}}\ar[ru]\ar[rd]\ar@{.}[ll]&&
			{\begin{smallmatrix}\circ\end{smallmatrix}}\ar[ru]\ar[rd]\ar@{.}[ll]&&
			{\begin{smallmatrix}\circ\end{smallmatrix}}\ar[ru]\ar[rd]\ar@{.}[ll]&&
			{\begin{smallmatrix}\circ\end{smallmatrix}}\ar[ru]\ar[rd]\ar@{.}[ll]&&
			{\begin{smallmatrix}\bullet\end{smallmatrix}}\ar[rd]\ar@{.}[ll]\\
			{\begin{smallmatrix}\bullet\end{smallmatrix}}\ar[ru]&&\textnormal{$\begin{smallmatrix}\bullet\end{smallmatrix}$}\ar[ru]\ar@{.}[ll]&&
			{\begin{smallmatrix}\circ\end{smallmatrix}}\ar[ru]\ar@{.}[ll]&&\textnormal{$\begin{smallmatrix}\circ\end{smallmatrix}$}\ar[ru]\ar@{.}[ll]&&
			{\begin{smallmatrix}\circ\end{smallmatrix}}\ar[ru]\ar@{.}[ll]&&
			{\begin{smallmatrix}\bullet\end{smallmatrix}}\ar[ru]\ar@{.}[ll]&&
			{\begin{smallmatrix}\circ\end{smallmatrix}}\ar[ru]\ar@{.}[ll]&&
			{\begin{smallmatrix}\circ\end{smallmatrix}}\ar[ru]\ar@{.}[ll]&&
			{\begin{smallmatrix}\bullet\end{smallmatrix}}\ar@{.}[ll]\\
	}}\notag
\end{align}
\begin{align}
	\tiny
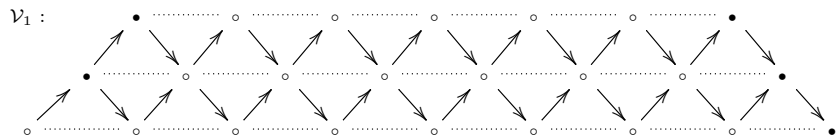
{\xymatrix @R=4mm @C2mm{
			\mathcal V_1:&&{\begin{smallmatrix}\bullet\end{smallmatrix}}\ar[rd]&&
			{\begin{smallmatrix}\circ\end{smallmatrix}}\ar[rd]\ar@{.}[ll]&&
			{\begin{smallmatrix}\circ\end{smallmatrix}}\ar[rd]\ar@{.}[ll]&&
			{\begin{smallmatrix}\circ\end{smallmatrix}}\ar[rd]\ar@{.}[ll]&&
			{\begin{smallmatrix}\circ\end{smallmatrix}}\ar[rd]\ar@{.}[ll]&&
			{\begin{smallmatrix}\circ\end{smallmatrix}}\ar[rd]\ar@{.}[ll]&&
			{\begin{smallmatrix}\bullet\end{smallmatrix}}\ar[rd]\ar@{.}[ll]&&\\
			&{\begin{smallmatrix}\bullet\end{smallmatrix}}\ar[ru]\ar[rd]&&
			{\begin{smallmatrix}\circ\end{smallmatrix}}\ar[ru]\ar[rd]\ar@{.}[ll]&&
			{\begin{smallmatrix}\circ\end{smallmatrix}}\ar[ru]\ar[rd]\ar@{.}[ll]&&
			{\begin{smallmatrix}\circ\end{smallmatrix}}\ar[ru]\ar[rd]\ar@{.}[ll]&&
			{\begin{smallmatrix}\circ\end{smallmatrix}}\ar[ru]\ar[rd]\ar@{.}[ll]&&
			{\begin{smallmatrix}\circ\end{smallmatrix}}\ar[ru]\ar[rd]\ar@{.}[ll]&&
			{\begin{smallmatrix}\circ\end{smallmatrix}}\ar[ru]\ar[rd]\ar@{.}[ll]&&
			{\begin{smallmatrix}\bullet\end{smallmatrix}}\ar[rd]\ar@{.}[ll]\\
			{\begin{smallmatrix}\circ\end{smallmatrix}}\ar[ru]&&\textnormal{$\begin{smallmatrix}\circ\end{smallmatrix}$}\ar[ru]\ar@{.}[ll]&&
			{\begin{smallmatrix}\circ\end{smallmatrix}}\ar[ru]\ar@{.}[ll]&&\textnormal{$\begin{smallmatrix}\circ\end{smallmatrix}$}\ar[ru]\ar@{.}[ll]&&
			{\begin{smallmatrix}\circ\end{smallmatrix}}\ar[ru]\ar@{.}[ll]&&
			{\begin{smallmatrix}\circ\end{smallmatrix}}\ar[ru]\ar@{.}[ll]&&
			{\begin{smallmatrix}\circ\end{smallmatrix}}\ar[ru]\ar@{.}[ll]&&
			{\begin{smallmatrix}\circ\end{smallmatrix}}\ar[ru]\ar@{.}[ll]&&
			{\begin{smallmatrix}\bullet\end{smallmatrix}}\ar@{.}[ll]\\
	}}\notag
\end{align}
\begin{align}
	\tiny
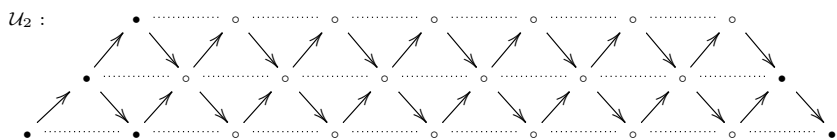
{\xymatrix @R=4mm @C2mm{
			\mathcal U_2:&&{\begin{smallmatrix}\bullet\end{smallmatrix}}\ar[rd]&&
			{\begin{smallmatrix}\circ\end{smallmatrix}}\ar[rd]\ar@{.}[ll]&&
			{\begin{smallmatrix}\circ\end{smallmatrix}}\ar[rd]\ar@{.}[ll]&&
			{\begin{smallmatrix}\circ\end{smallmatrix}}\ar[rd]\ar@{.}[ll]&&
			{\begin{smallmatrix}\circ\end{smallmatrix}}\ar[rd]\ar@{.}[ll]&&
			{\begin{smallmatrix}\circ\end{smallmatrix}}\ar[rd]\ar@{.}[ll]&&
			{\begin{smallmatrix}\circ\end{smallmatrix}}\ar[rd]\ar@{.}[ll]&&\\
			&{\begin{smallmatrix}\bullet\end{smallmatrix}}\ar[ru]\ar[rd]&&
			{\begin{smallmatrix}\circ\end{smallmatrix}}\ar[ru]\ar[rd]\ar@{.}[ll]&&
			{\begin{smallmatrix}\circ\end{smallmatrix}}\ar[ru]\ar[rd]\ar@{.}[ll]&&
			{\begin{smallmatrix}\circ\end{smallmatrix}}\ar[ru]\ar[rd]\ar@{.}[ll]&&
			{\begin{smallmatrix}\circ\end{smallmatrix}}\ar[ru]\ar[rd]\ar@{.}[ll]&&
			{\begin{smallmatrix}\circ\end{smallmatrix}}\ar[ru]\ar[rd]\ar@{.}[ll]&&
			{\begin{smallmatrix}\circ\end{smallmatrix}}\ar[ru]\ar[rd]\ar@{.}[ll]&&
			{\begin{smallmatrix}\bullet\end{smallmatrix}}\ar[rd]\ar@{.}[ll]\\
			{\begin{smallmatrix}\bullet\end{smallmatrix}}\ar[ru]&&\textnormal{$\begin{smallmatrix}\bullet\end{smallmatrix}$}\ar[ru]\ar@{.}[ll]&&
			{\begin{smallmatrix}\circ\end{smallmatrix}}\ar[ru]\ar@{.}[ll]&&\textnormal{$\begin{smallmatrix}\circ\end{smallmatrix}$}\ar[ru]\ar@{.}[ll]&&
			{\begin{smallmatrix}\circ\end{smallmatrix}}\ar[ru]\ar@{.}[ll]&&
			{\begin{smallmatrix}\circ\end{smallmatrix}}\ar[ru]\ar@{.}[ll]&&
			{\begin{smallmatrix}\circ\end{smallmatrix}}\ar[ru]\ar@{.}[ll]&&
			{\begin{smallmatrix}\circ\end{smallmatrix}}\ar[ru]\ar@{.}[ll]&&
			{\begin{smallmatrix}\bullet\end{smallmatrix}}\ar@{.}[ll]\\
	}}\notag
\end{align}
\begin{align}
	\tiny
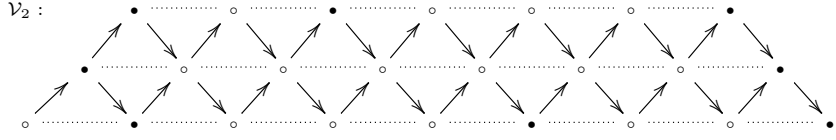
{\xymatrix @R=4mm @C2mm{
			\mathcal V_2:&&{\begin{smallmatrix}\bullet\end{smallmatrix}}\ar[rd]&&
			{\begin{smallmatrix}\circ\end{smallmatrix}}\ar[rd]\ar@{.}[ll]&&
			{\begin{smallmatrix}\bullet\end{smallmatrix}}\ar[rd]\ar@{.}[ll]&&
			{\begin{smallmatrix}\circ\end{smallmatrix}}\ar[rd]\ar@{.}[ll]&&
			{\begin{smallmatrix}\circ\end{smallmatrix}}\ar[rd]\ar@{.}[ll]&&
			{\begin{smallmatrix}\circ\end{smallmatrix}}\ar[rd]\ar@{.}[ll]&&
			{\begin{smallmatrix}\bullet\end{smallmatrix}}\ar[rd]\ar@{.}[ll]&&\\
			&{\begin{smallmatrix}\bullet\end{smallmatrix}}\ar[ru]\ar[rd]&&
			{\begin{smallmatrix}\circ\end{smallmatrix}}\ar[ru]\ar[rd]\ar@{.}[ll]&&
			{\begin{smallmatrix}\circ\end{smallmatrix}}\ar[ru]\ar[rd]\ar@{.}[ll]&&
			{\begin{smallmatrix}\circ\end{smallmatrix}}\ar[ru]\ar[rd]\ar@{.}[ll]&&
			{\begin{smallmatrix}\circ\end{smallmatrix}}\ar[ru]\ar[rd]\ar@{.}[ll]&&
			{\begin{smallmatrix}\circ\end{smallmatrix}}\ar[ru]\ar[rd]\ar@{.}[ll]&&
			{\begin{smallmatrix}\circ\end{smallmatrix}}\ar[ru]\ar[rd]\ar@{.}[ll]&&
			{\begin{smallmatrix}\bullet\end{smallmatrix}}\ar[rd]\ar@{.}[ll]\\
			{\begin{smallmatrix}\circ\end{smallmatrix}}\ar[ru]&&\textnormal{$\begin{smallmatrix}\bullet\end{smallmatrix}$}\ar[ru]\ar@{.}[ll]&&
			{\begin{smallmatrix}\circ\end{smallmatrix}}\ar[ru]\ar@{.}[ll]&&\textnormal{$\begin{smallmatrix}\circ\end{smallmatrix}$}\ar[ru]\ar@{.}[ll]&&
			{\begin{smallmatrix}\circ\end{smallmatrix}}\ar[ru]\ar@{.}[ll]&&
			{\begin{smallmatrix}\bullet\end{smallmatrix}}\ar[ru]\ar@{.}[ll]&&
			{\begin{smallmatrix}\circ\end{smallmatrix}}\ar[ru]\ar@{.}[ll]&&
			{\begin{smallmatrix}\circ\end{smallmatrix}}\ar[ru]\ar@{.}[ll]&&
			{\begin{smallmatrix}\bullet\end{smallmatrix}}\ar@{.}[ll]\\
	}}\notag
\end{align}

\end{example}

\begin{definition}
Let $(\X,\Y)$ be an $n$-cotorsion pair in $\C$. We call $\I(\X)=\X\cap\Y$ the $core$ of $(\X,\Y)$.
\end{definition}

\begin{remark}\label{l}
By the definition of $n$-cotorsion pair, it is easy to see the core $\I(\X)$ is an $(n+1)$-rigid subcategory.
\end{remark}

\subsection{Compatibility with subfactor triangulated categories}
In this subsection, we fix a functorially finite $(n+1)$-rigid subcategory $\D$ of $\C$ satisfying the condition
\[
\bigcap\limits_{k=1}^{n}{}^{\perp_k}\D = \bigcap\limits_{k=1}^{n}\D^{\perp_k},
\]
and denote this common subcategory by $\Z$. The quotient category $\mathfrak{U} := \Z/\D$ is called a \emph{subfactor triangulated category}, and it is defined as follows:

\begin{itemize}
  \item[(1)] The objects of $\mathfrak{U}$ are the same as those in $\Z$.
  \item[(2)] For each $X,Y \in \Z$, the morphism space is defined by
  \[
  \Hom_{\mathfrak{U}}(X,Y) := \Hom_{\Z}(X,Y)/\D(X,Y),
  \]
  where $\D(X,Y)$ denotes the subgroup of morphisms in $\Hom_{\Z}(X,Y)$ that factor through some object in $\D$.
\end{itemize}

It is easy to verify that the subcategory $\Z$ satisfies the following properties:

\begin{itemize}
  \item[(1)] $\Z$ is extension-closed in $\C$; that is, for any $\E$-triangle in $\C$
  \[
  \xymatrix{A\ar[r]^{f}& B\ar[r]^{g}& C\ar@{-->}[r]^{\delta}&,}
  \]
  if $A,C \in \Z$, then $B \in \Z$ as well.

  \item[(2)] The pair $(\Z, \Z)$ forms a $\D$-mutation pair.
\end{itemize}

It is shown in \cite[Theorem 3.13]{ZZ1} that $\mathfrak{U}$ admits a triangulated structure, where the shift functor is denoted by $\langle1\rangle$, and triangles are defined as follows:

\begin{itemize}
  \item[$\bullet$] Since $(\Z, \Z)$ is a $\D$-mutation pair, for each object $X \in \mathfrak{U}$, we may choose an $\E$-triangle
  \[
  \xymatrix{X \ar[r]^{f} & D \ar[r]^{g} & X' \ar@{-->}[r]^{\delta} &,}
  \]
  where $D \in \D$. The shift of $X$ in $\mathfrak{U}$ is defined to be $X'$, and we denote it by $X\langle1\rangle$.

  \item[$\bullet$] Given an $\E$-triangle in $\C$
  \[
  \xymatrix{X \ar[r]^{a} & Y \ar[r]^{b} & Z \ar@{-->}[r]^{\delta'} &,}
  \]
  with $X,Y,Z \in \Z$, since $\E(Z, D) = 0$ for all $D \in \D$, there exists a morphism $d\colon Z \rightarrow X\langle1\rangle$ making the following diagram commute:
  \[
  \xymatrix@C=1cm{
    X \ar@{=}[d] \ar[r]^{a} & Y \ar[d] \ar[r]^{b} & Z \ar[d]^{d} \ar@{-->}[r]^{\delta'} & \\
    X \ar[r]^{f} & D \ar[r]^{g} & X\langle1\rangle \ar@{-->}[r]^{\delta} &.
  }
  \]
  This gives rise to a complex in $\mathfrak{U}$:
  \[
  X \xrightarrow{\bar{a}} Y \xrightarrow{\bar{b}} Z \xrightarrow{\bar{d}} X\langle1\rangle,
  \]
  and the triangles in $\mathfrak{U}$ are defined as those complexes obtained in this way.
\end{itemize}
\vspace{2mm}

Let $\langle i \rangle$ denote the $i$-th shift functor in the triangulated category $\mathfrak{U}$. The following result will be used.

\begin{lemma}\label{c}
Let $\D$ be a functorially finite $(n+1)$-rigid subcategory of $\C$ satisfying the condition
\[
\bigcap\limits_{k=1}^{n}{}^{\perp_k}\D = \bigcap\limits_{k=1}^{n}\D^{\perp_k}.
\]
Then for any $X, Y \in \Z$, there exists an isomorphism
\[
\mathfrak{U}(X, Y\langle i \rangle) \cong \E^i(X, Y)
\]
for all integers $1 \leq i \leq n$.
\end{lemma}

\begin{proof}
Since $(\Z, \Z)$ is a $\D$-mutation pair, for any object $Y\in\mathfrak{U}$, take a chain of $\E$-triangles
    $$\begin{array}{r}
     \xymatrix{Y\ar[r]^{f_1}& D_1\ar[r]^{g_1\;\;}&Y\langle1\rangle\ar@{-->}[r]&,}\\
     \xymatrix{Y\langle1\rangle\ar[r]^{f_2}& D_2\ar[r]^{g_2\;\;}&Y\langle2\rangle\ar@{-->}[r]&,}\\
\vdots\hspace{3.6cm}\\
    \xymatrix{
     Y\langle n-1\rangle\ar[r]^{\quad f_n}& D_n\ar[r]^{g_n\;\;}&Y\langle n\rangle\ar@{-->}[r]&,
     }
     \end{array}$$
where $D_i\in\D$, $f_i$ is a left $\D$-approximation and $g_i$ is a right $\D$-approximation for $1\leq i\leq n$. Applying the functor $\Hom_{\C}(X,-)$ to the above $\E$-triangles, we can get the following long exact sequence by Lemma \ref{2}
$$\Hom_{\C}(X,Y\langle i-1\rangle)\xrightarrow{~(f_i)_{\ast}~}\Hom_{\C}(X,D_i)\xrightarrow{~(g_i)_{\ast}~}\Hom_{\C}(X,Y\langle i\rangle)\rightarrow\E(X,Y\langle i-1\rangle)\rightarrow$$
$$\E(X,D_i)\rightarrow\cdots\rightarrow\E^k(X,Y\langle i-1\rangle)\rightarrow\E^k(X,D_i)\rightarrow\cdots.$$
Since $X\in\Z$ and $D_i\in\D$, we have $\E^k(X,D_i)=0$ for $1\leq k\leq n$. Thus $\E^k(X,Y\langle i-1\rangle)\cong\E^{k-1}(X,Y\langle i\rangle)$ holds for $2\leq k\leq n$, $1\leq i\leq n$.

Since $g_i$ is a right $\D$-approximation, one can check that Im$(g_i)_{\ast}=[\D](X,Y\langle i\rangle)$ holds for $1\leq i\leq n$. So
$$\mathfrak{U}(X,Y\langle i\rangle)=\Hom_{\C}(X,Y\langle i\rangle)/[\D](X,Y\langle i\rangle)=\Hom_{\C}(X,Y\langle i\rangle)/\text{Im}(g_i)_{\ast}\cong$$
$$\E(X,Y\langle i-1\rangle)\cong\E^2(X,Y\langle i-2\rangle)\cong\cdots\cong\E^i(X,Y)$$
 for any $1\leq i\leq n$.
\end{proof}
\begin{lemma}\label{b}
Let $(\X,\Y)$ be an $n$-cotorsion pair in $\C$ with core $\I(\X)$. Then $\D\subset\I(\X)$ if and only if $\D\subset\X\subset\Z$ and $\D\subset\Y\subset\Z$.
\end{lemma}
\begin{proof}
The `if' part is trivial since  $\I(\X)=\X\cap\Y$. To prove the `only if' part, let $X$ be an arbitrary object in $\X$. Since  $\D\subset\I(\X)\subset\Y$, we have $\E^i(X,\D)=0$  for $1\leq i\leq n$. It follows that $X\in\Z$ and hence $\X\subset\Z$. Similarly, we have $\Y\subset\Z$.
\end{proof}

Building on the previous lemma, we study the connection between $n$-cotorsion pairs in $\C$ whose cores contain $\D$, and those in the subfactor triangulated category $\mathfrak{U}$. In what follows, let $\W$ be a subcategory of $\C$ satisfying $\D \subseteq \W \subseteq \Z$. We use $\overline{\W}$ to denote the corresponding subcategory of $\Z$. Clearly, every subcategory of $\Z$ can be expressed in this way.

Let $\C$ be an extriangulated category, and let $\X$ and $\Y$ be subcategories of $\C$. We define the right and left orthogonal subcategories as
\[
\X^\perp := \{ C \in \C \mid \Hom_\C(\X, C) = 0 \} \quad \text{and} \quad {}^\perp\X := \{ C \in \C \mid \Hom_\C(C, \X) = 0 \}.
\]
We write $\X \ast \Y$ for the class of objects in $\C$ that can be realized as the middle terms of $\mathbb{E}$-triangles of the form
\[
\xymatrix{X \ar[r]^{f} & C \ar[r]^{g} & Y \ar@{-->}[r] &}\;\; \text{with } X \in \X \text{ and } Y \in \Y.
\]

To prove the main results of this section, we need the following preparations.

\begin{definition}
Suppose that $\C$ is an extriangulated category. We call a pair $(\X,\Y)$ of subcategories of $\C$ a \emph{torsion} \emph{pair} if
$$\Hom_\C(\X,\Y)=0\;\;\text{and}\;\;\C=\X\ast\Y.$$
\end{definition}

We have the following Wakamatsu-type lemma.

\begin{lemma}\label{h}
\emph{(1)} Let $\cal X$ be a contravariantly finite and extension closed subcategory of an extriangulated category $\C$. Then $(\cal X,\cal X^\bot)$ is a torsion pair.

\emph{(2)} Let $\cal X$ be a covariantly finite and extension closed subcategory of an extriangulated category $\C$. Then  $(^\bot{\cal X},\cal X)$ is a torsion pair.
\end{lemma}

\begin{proof}
(1) It is obvious that $\Hom(\cal X,\cal X^\bot)=0$.\\
Assume that $C$ is an object in $\C$, since $\X$ is contravariantly finite,  take an $\E$-triangle
$$\xymatrix{X\ar[r]^f&C\ar[r]^g&Y\ar@{-->}[r]&}$$
with a right minimal $\cal X$-approximation $f$ of $C$. We only have to show $Y\in\cal X^\bot$. \\Take any morphism $u\in
\Hom(X',Y)=0$ with $X'\in\cal X$.
By \cite[Remark 2.22]{NP}, we have a commutative diagram
$$\xymatrix{X\ar@{=}[d]\ar[r]^m&A\ar[d]^v\ar[r]^h&X'\ar[d]^u\ar@{-->}[r]&\\
X\ar[r]^f&C\ar[d]\ar[r]^g&Y\ar[d]\ar@{-->}[r]&\\
                 &Z\ar@{-->}[d]\ar@{=}[r]&Z\ar@{-->}[d]\\&&}$$
of $\E$-triangles. Since $\cal X$ is an extension closed subcategory, we have $A\in\cal X$. Since $f$ is
a right $\cal X$-approximation of $C$, there is a morphism $s:A\rightarrow X$ such that $v=fs$.
Thus, $f=vm=fsm$. Since $f$ is a right minimal, we have $sm$ is an automorphism,
and $m$ is a section. So $h$ is a retraction. Hence $uh=gv=gfs=0$, and then $u=0$.

(2) This is dual of (1).
\end{proof}

\begin{lemma}\label{4}
Let $\M_i$ be a contravariantly finite and extension closed subcategory of an extriangulated category $\C$ such that
$\E(\M_i,\M_j)=0$ for any $i<j$. Put
$$\X_n\colon=\add(\M_1\ast\M_2\ast\cdots\ast\M_n),\;\;\Y_n=\colon\bigcap\limits_{k=1}^{n}\M_i^\perp.$$
Then $(\X_n,\Y_n)$ is a torsion pair.
\end{lemma}
\begin{proof}
The case $n=1$ is proved in Lemma \ref{h}. We show $\C=\X_n\ast\Y_n$ by induction on $n$. Assume that the assertion is true for $n=i-1$. For any $C\in\C$, take an $\E$-triangle
$$\xymatrix{X_{i-1}\ar[r]&C\ar[r]&Y_{i-1}\ar@{-->}[r]&}$$
with $X_{i-1}\in\X_{i-1}$ and $Y_{i-1}\in\Y_{i-1}$. For the object $Y_{i-1}$, take an $\E$-triangle
$$\xymatrix{M_i\ar[r]^{a_i}&Y_{i-1}\ar[r]&Y_{i}\ar@{-->}[r]&}$$
with a minimal right $\M_i$-approximation of $Y_{i-1}$. By Lemma \ref{h}, $Y_i\in\M_i^\perp$. Next, we show $Y_i\in\Y_{i-1}$. Applying $\Hom_{\C}(\M_k,-)$ for $1\leq k\leq i-1$ to the above $\E$-triangle, we get a long exact sequence by Lemma \ref{2}
$$\Hom_{\C}(\M_k,\M_i)\rightarrow\Hom_{\C}(\M_k,Y_{i-1})\rightarrow\Hom_{\C}(\M_k,Y_i)\rightarrow\E^1(\M_k,\M_i)\rightarrow\cdots.$$
Since $Y_{i-1}\in\Y_{i-1}$, we have $\Hom_{\C}(\M_k,Y_{i-1})=0$ for $1\leq k\leq i-1$. Moreover, we have $\E(\M_k,\M_i)=0$ by assumption, so $\Hom_{\C}(\M_k,Y_i)=0$ holds for $1\leq k\leq i-1$. Thus we show $Y_i\in\Y_{i-1}$. That means $Y_i\in\Y_{i-1}\cap\M_i^\perp=\Y_i$. By \cite[Remark 2.22]{NP}, we have a commutative diagram
$$\xymatrix{X_{i-1}\ar@{=}[d]\ar[r]^m&X_i\ar[d]^v\ar[r]^h&M_i\ar[d]^{a_i}\ar@{-->}[r]&\\
X_{i-1}\ar[r]&C\ar[d]\ar[r]&Y_{i-1}\ar[d]\ar@{-->}[r]&\\
                 &Y_i\ar@{-->}[d]\ar@{=}[r]&Y_i\ar@{-->}[d] \\&&}$$
of $\E$-triangles. Since $X_i\in\X_{i-1}\ast\M_i\subset\X_i$, we have $C\in\X_i\ast\Y_i$ by the second column $\E$-triangle.
\end{proof}

\begin{lemma}\label{a}
Let $\X$ be a subcategory of $\C$ satisfying  $\D\subset\X\subset\Z$. Then
\begin{itemize}
  \item [\rm (1)] $\X$ is a contravariantly finite subcategory of $\C$ if and only if $\overline{\X}$ is a contravariantly finite subcategory of $\mathfrak{U}$.
  \item [\rm (2)] $\X$ is a covariantly finite subcategory of $\C$ if and only if $\overline{\X}$ is a covariantly finite subcategory of $\mathfrak{U}$.
  \item [\rm (3)] $\X$ is a functorially finite subcategory of $\C$ if and only if $\overline{\X}$ is a functorially finite subcategory of $\mathfrak{U}$.
\end{itemize}
\end{lemma}

\begin{proof}
We only prove the first assertion, the others can be proved similarly.

Assume that $\X$ is a contravariantly finite subcategory of $\C$. Let $Z$ be an object in $\Z$. Since $\X$ is contravariantly finite, there exists an object $X \in \X$ and a morphism $f \colon X \to Z$ which serves as a right $\X$-approximation of $Z$ in $\C$. It is straightforward to verify that the induced morphism $\overline{f} \colon X \to Z$ is a right $\overline{\X}$-approximation of $Z$ in the quotient category $\mathfrak{U}$.

Conversely, assume that $\overline{\X}$ is a contravariantly finite subcategory of $\mathfrak{U}$. Since $\D$ is functorially finite in $\C$, it follows from Lemma~\ref{4} and its dual that $\Z$ is also functorially finite in $\C$. Let $C$ be an object in $\C$. Then, as $\Z$ is functorially finite, there exists an object $Z \in \Z$ along with a morphism $f \colon Z \to C$ that is a right $\Z$-approximation of $C$.

By assumption, there exists $X \in \X$ and a morphism $\overline{g} \colon X \to Z$ in $\mathfrak{U}$ such that it is a right $\overline{\X}$-approximation of $Z$. Moreover, since $\D$ is functorially finite in $\C$ and $Z \in \Z \subseteq \C$, there exists an object $D \in \D$ with a morphism $c \colon D \to Z$ forming a right $\D$-approximation of $Z$. We claim that the morphism
\[
(f \circ g,\, f \circ c) \colon X \oplus D \to C
\]
is a right $\X$-approximation of $C$ in $\C$.

To verify the claim, consider any morphism $h \colon X' \to C$ with $X' \in \X \subseteq \Z$. Since $f \colon Z \to C$ is a right $\Z$-approximation, there exists $\ell \colon X' \to Z$ such that $h = f \circ \ell$. As $\overline{g} \colon X \to Z$ is a right $\overline{\X}$-approximation in $\mathfrak{U}$, there exists a morphism $\overline{k} \colon X' \to X$ such that $\overline{\ell} = \overline{g} \circ \overline{k}$ and then
$
\overline{\ell - g \circ k} = \overline{0}.
$
Hence, the morphism $\ell - g \circ k$ factors through some $D_1 \in \D$; that is, there exist morphisms $a_1 \colon X' \to D_1$ and $a_2 \colon D_1 \to Z$ such that $\ell - g \circ k = a_2 \circ a_1$. Since $c \colon D \to Z$ is a right $\D$-approximation, there exists $d \colon D_1 \to D$ such that $a_2 = c \circ d$. Consequently,
\[
\ell = g \circ k + c \circ d \circ a_1 \quad \text{and} \quad h = f \circ \ell = f \circ (g \circ k + c \circ d \circ a_1).
\]
This shows that $h$ factors through the morphism $(f \circ g, f \circ c) \colon X \oplus D \to C$ via
\[
X' \xrightarrow{\binom{k}{d \circ a_1}} X \oplus D.
\]
Thus, we conclude that $(f \circ g, f \circ c)$ is indeed a right $\X$-approximation of $C$ in $\C$, and therefore, $\X$ is contravariantly finite.
\end{proof}

Thus, one can consider the relationship between cotorsion pairs in $\C$ whose cores contain
$\D$ and cotorsion pairs in the subfactor triangulated category $\mathfrak{U}$.

\begin{theorem}\label{d}
Let $(\X,\Y)$ be an $n$-cotorsion pair in $\C$ with $\D\subset\I(\X)$. Then $(\overline{\X},\overline{\Y})$ is an $n$-cotorsion pair in $\mathfrak{U}$ with $\I(\overline{\X})=\overline{\I(\X)}$. Moreover, the map $(\X,\Y)\mapsto(\overline{\X},\overline{\Y})$ is a bijection from the set of $n$-cotorsion pair in $\C$ whose cores contain $\D$ to the set of $n$-cotorsion pair in $\mathfrak{U}$.
\end{theorem}

\begin{proof}
Since $(\X,\Y)$ is an $n$-cotorsion pair in $\C$ with $\D\subset\I(\X)$, we have $\D\subset\X\subset\Z$ and $\D\subset\Y\subset\Z$ by Lemma \ref{b}. So $\overline{\X}$ and $\overline{\Y}$ are subcategories of $\mathfrak{U}$. By Lemma \ref{c}, there is an isomorphism $\mathfrak{U}(\X,\Y\langle i\rangle)\cong\E^i(\X,\Y)$ for any $1\leq i\leq n$, so $\X=\bigcap\limits_{i=1}^{n}{^{\bot_i}}\Y$ for all $1\leq i\leq n$ if and only if $\overline{\X}=\bigcap\limits_{i=1}^{n}{^\bot}\overline{\Y}\langle i\rangle$ for all $1\leq i\leq n$. Moreover, By Lemma \ref{a}, we get $(\overline{\X},\overline{\Y})$ is an $n$-cotorsion pair in $\mathfrak{U}$. In this case, we have  $\I(\overline{\X})=\overline{\X}\cap\overline{\Y}=\overline{\X\cap\Y}=\overline{\I(\X)}$.

The bijection from the set of $n$-cotorsion pair in $\C$ whose cores contain $\D$ to the set of $n$-cotorsion pair in $\Z$ is obvious by Lemma \ref{b}, Lemma \ref{c} and Lemma \ref{a}.
\end{proof}

This theorem immediately yields the following important conclusion.

\begin{corollary} \label{cor11}
The correspondence $\X\mapsto \overline{\X}$ gives
a one-one correspondence between $(n+1)$-cluster tilting subcategories of $\C$
containing $\D$ and $(n+1)$-cluster tilting subcategories of $\U$.
\end{corollary}

\begin{remark}
In the case where $\mathcal{C}$ is a triangulated category, Corollary \ref{cor11} coincides with \cite[Theorem 4.9(1)]{IY}.
\end{remark}

\begin{remark}
When $n=1$, Theorem~\ref{d} recovers the classical case and
 coincides with \cite[Theorem 3.2]{ZZ}.
\end{remark}

\subsection{Mutation of $n$-cotorsion pairs}
Let $\C$ be an extriangulated category with enough projectives and enough injectives. In this subsection, we study mutation of $n$-cotorsion pairs.

\begin{definition}\label{e}
Fix a functorially finite $(n+1)$-rigid subcategory $\D$ of $\C$. For a subcategory $\X$ of $\C$, let $\mu^{-1}_{\D}(\X)$ be a subcategory of $\C$ consisting of all $M\in\C$ such that there exists an $\E$-triangle
 $$\xymatrix{X\ar[r]^{f}& D\ar[r]&M\ar@{-->}[r]&}$$
with $X\in\X$ and a left $\D$-approximation $f$. Dually, for a subcategory $\Y$ of $\C$, let $\mu_{\D}(\Y)$ a subcategory of $\C$ consisting of all $M\in\C$ such that there exists an $\E$-triangle
 $$\xymatrix{M\ar[r]& D\ar[r]^{g}&Y\ar@{-->}[r]&}$$
with  $Y\in\Y$ and a right $\D$-approximation  $g$.

 In this case, $\mu^{-1}_{\D}(\X)$ is called the $forward$ $\D$-mutation of $\X$ and $\mu_{\D}(\Y)$ is called the $backward$ $\D$-mutation of $\Y$.
\end{definition}

\begin{remark}
Note that our definition is a natural generalization of Definition~\ref{3}. When $\C$ is a triangulated category, it coincides with Definition~3.13 in \cite{CZ}. In particular, when $n=1$, our definition of mutation recovers the classical Iyama-Yoshino mutation \cite{IY}.
\end{remark}

\begin{proposition}
Let $\D$ be a functorially finite $(n+1)$-rigid subcategory of $\C$  satisfying $$\bigcap\limits_{k=1}^{n}{^{\perp_k}}\D=\bigcap\limits_{k=1}^{n}\D^{\perp_k}.$$
 Then the maps $\mu^{-1}_{\D}(-)$ and $\mu_{\D}(-)$ are mutually inverse on the set of subcategories $\M$ of $\C$ satisfying $\D\subset\M\subset\Z$.
\end{proposition}

\begin{proof}
Let $M$ be an object of $M$. By Definition \ref{e}, $\mu_{\D}(M)=X$ such that there exists an $\E$-triangle $$\xymatrix{X\ar[r]^{f}& D\ar[r]^{g}&M\ar@{-->}[r]&}$$
with a right $\D$-approximation $g$. Since $M\in\Z$, we have $\E^1(M,D)=0$. So $f$ is a left $\D$-approximation of $X$. This shows $\mu^{-1}_{\D}(\mu_{\D}(M))=M$.
\end{proof}

The following theorem shows that any mutation of an $n$-cotorsion pair
is again an $n$-cotorsion pair.

\begin{theorem}\label{main}
Let $(\X,\Y)$ be an $n$-cotorsion pair in $\C$ and $\D\subset\I(\X)$ be a functorially finite subcategory of $\C$ satisfying  $\bigcap\limits_{k=1}^{n}{^{\perp_k}}\D=\bigcap\limits_{k=1}^{n}\D^{\perp_k}$. Then the pairs $$(\mu^{-1}_{\D}(\X),\mu^{-1}_{\D}(\Y))\hspace{2mm}\mbox{and}\hspace{3mm}(\mu_{\D}(\X),\mu_{\D}(\Y))$$ are  $n$-cotorsion pairs in $\C$ and we have
$$\I(\mu^{-1}_{\D}(\X))=\mu^{-1}_{\D}(\I(\X))\hspace{2mm}\mbox{and}\hspace{3mm}\I(\mu_{\D}(\X))=\mu_{\D}(\I(\X)).$$
\end{theorem}
\begin{proof}
By Remark \ref{l}, we have that $\I(\X)$ is $(n+1)$-rigid, then so is $\D$. So both $\mu^{-1}_{\D}(\X)$ and $\mu^{-1}_{\D}(\Y)$ are well-defined. We denote $\mu^{-1}_{\D}(\X)$, $\mu^{-1}_{\D}(\Y)$ and $\mu^{-1}_{\D}(\I(\X))$ by
$\X^\prime$,  $\Y^\prime$ and $\I^\prime$ respectively. As in the previous subsection, we denote by
$$\Z=\bigcap\limits_{i=1}^{n}\D^{\perp_i}=\bigcap\limits_{i=1}^{n}{^{\perp_i}}\D$$ and $\mathfrak{U}=\Z/\D$ the subfactor triangulated category. By Theorem \ref{d}, we have that $(\overline{\X},\overline{\Y})$ is an $n$-cotorsion pair in $\mathfrak{U}$ and $\I(\overline{\X})=\overline{\I(\X)}$. Note that its  forward 0-mutation $(\overline{\X}\langle1\rangle,\overline{\Y}\langle1\rangle)$ is also an $n$-cotorsion pair in $\mathfrak{U}$. By the definition of forward (backword) mutation, we have that $\overline{\X}\langle1\rangle=\overline{\X^\prime}$ and $\overline{\Y}\langle1\rangle=\overline{\Y^\prime}$. Then by Theorem \ref{d}, we have that the pair $(\X^\prime,\Y^\prime)$ is an $n$-cotorsion pair in $\C$ and $\I(\overline{\X^\prime})=\overline{\I(\X^\prime)}$. This means that $ \overline{\I^\prime}=\overline{\I(\X)}\langle1\rangle=\I(\overline{\X})\langle1\rangle=
\I(\overline{\X^\prime})=\overline{\I(\X^\prime)}$. Therefore, $\I(\mu^{-1}_{\D}(\X))=\mu^{-1}_{\D}(\I(\X))$.

The assertion for $\I(\mu_{\D}(\X))=\mu_{\D}(\I(\X))$ can be proved similarly.
\end{proof}

\begin{remark}
When $n=1$, Theorem~\ref{main} recovers the classical case and
 coincides with \cite[Theorem 4.3]{ZZ}.
\end{remark}

\section{$n$-cotorsion pairs in $n$-cluster categories of type $A_{\infty}$}
In this section, we provide a geometric description of  $n$-cotorsion pairs in the $n$-cluster category of type $A_{\infty}$, denoted by $\C_n(A_{\infty})$, using specific configurations of $n$-admissible arcs. When $n = 1$, our result recovers and extends the classification of classical cotorsion pairs in cluster categories of type $A_{\infty}$ previously established by Ng~\cite{Ng} and Zhou--Zhou~\cite{ZZ}. As an application, we present a geometric realization of the mutation of $n$-cotorsion pairs through rotations of the corresponding configurations of $n$-admissible arcs.

\subsection{Geometric realization of $\C_n(A_{\infty})$}
Let $k$ be an algebraically closed field, $n$ an integer, and let $\T$ be a $k$-linear algebraic triangulated category that is idempotent complete and classically generated by an $(n+1)$-spherical object. This category was first studied in \cite{J} for the case $n = 1$.

For different values of $n$, the category $\T$ appears in various well-known forms. When $n = -1$, $\T$ is equivalent to $D^c(k[X]/(X^2))$, the compact derived category of the dual numbers. When $n = 0$, it corresponds to $D^f(k[![X]!])$, the derived category of complexes with bounded finite-length homology over the formal power series ring. For $n = 1$, $\T$ is the cluster category of type $A_{\infty}$; see \cite{HJ}. When $n > 1$, it can be interpreted as the $n$-cluster category of type $A_{\infty}$; see \cite{HJ1}. On the other hand, for negative values of $n$, the category $\T$ is less commonly studied. We now recall the definition of $\T$.

\begin{definition}\cite[Definition 1.1]{HJ1}
We say that $\T$ is a $k$-linear algebraic triangulated category with suspension functor $\Sigma$, which is idempotent complete and classically generated by an $(n+1)$-spherical object $s$, if
\[
\dim_k\T(s,\Sigma^{\ell} s)=\left\{
\begin{array}{cl}
1  & \text{if\ $\ell = 0$ or $\ell = n+1$,} \\
0  & \text{otherwise.}
\end{array}
\right.
\]

\end{definition}

From now on, suppose that $n\geq1$ is an integer and $\C_n(A_{\infty})$ is the $n$-cluster categories of type $A_{\infty}$. By \cite[Proposition 1.8]{HJY}, the Serre functor of $\C_n(A_{\infty})$ is $S=\Sigma^{n+1}$, and by \cite[Proposition 1.10]{HJY}, the Auslander-Reiten quiver  of $\C_n(A_{\infty})$
consists of $n$ components, each of which is a copy of $\mathbb{Z}A_{\infty}$, and $\Sigma$ acts cyclically
on the set of components. The $n$ components are denoted by $R,\Sigma R,\cdots,\Sigma^{n-1} R$ and satisfies $\Sigma^{n+i} R=\Sigma^{i} R$ for $i=0,1,\cdots,n-1$.

We recall the geometric realization of $\C_n(A_{\infty})$ based on \cite{HJ1}.
\begin{definition}
\begin{itemize}
  \item[(1)] Let $\Pi$ be an $\infty$-gon. A pair of integers $(t, u)$ as an arc in $\Pi$
connecting the integers $t$ and $u$ with $u-t\geq2$ and $u-t\equiv 1(\mod n)$ is
called an $n$-\emph{admissible arc}.
  \item[(2)] Two  $n$-admissible arcs $(r, s)$ and $(t, u)$ \emph{cross} if $r<t<s<u$ or $t<r<u<s$.
  \item[(3)] Let $\mathfrak{U}$ be a set of $n$-admissible arcs. An integer $t$ is a \emph{left-fountain} of $\mathfrak{U}$ if $\mathfrak{U}$ contains infinitely many $n$-admissible arcs of the form $(s, t)$, and $t$ is a \emph{right-fountain} of $\mathfrak{U}$ if $\mathfrak{U}$ contains infinitely many $n$-admissible arcs of the form $(t, u)$. We say that $t$ is a \emph{fountain} of $\mathfrak{U}$ if it is both a left-fountain and a right-fountain of $\mathfrak{U}$.
\end{itemize}
\end{definition}
The suspension functor $\Sigma$, the Serre funcotor $S$, and the AR translation $\tau=S\Sigma^{-1}$ of $\C_n(A_{\infty})$ act on an $n$-admissible arc $(t, u)$ are given by
$$\Sigma(t, u)=(t-1,u-1),\;\;\;S(t, u)=(t-n-1,u-n-1),\;\;\;\tau(t, u)=(t-n,u-n).$$

\begin{proposition}\cite[Proposition 2.4]{HJ1}
There is a bijective correspondence between $\ind \C_n(A_{\infty})$ and the set of $n$-admissible arcs. Moreover, This extends to a bijective correspondence between subcategories of $\C_n(A_{\infty})$ and subsets of the set of $n$-admissible arcs.
\end{proposition}

For a subcategory $\mathcal{X}$ of $\mathcal{C}_n(A_{\infty})$, we denote the corresponding set of $n$-admissible arcs in $\Pi$ by $\mathfrak{X}$. Conversely, for a set of $n$-admissible arcs $\mathfrak{X}$ in $\Pi$, we denote the corresponding subcategory of $\mathcal{C}_n(A_{\infty})$ by $E(\mathfrak{X})$. When there is no confusion, we sometimes use an $n$-admissible arc $(t, u)$ to represent an indecomposable object $M$ in $\mathcal{C}_n(A_{\infty})$, and write $M = (t, u)$.

\begin{example}
Let $n=3$. The Auslander-Reiten quiver  of $\C_3(A_{\infty})$ contains 3 components $R,\Sigma R$ and $\Sigma^2R$. We draw the  Auslander-Reiten quiver  of  $\C_3(A_{\infty})$, see Figures \ref{figure1}-\ref{figure3}.
\begin{figure}[h]
\[
  \xymatrix @-3.0pc @! {
    & \vdots \ar[dr] & & \vdots \ar[dr] & & \vdots \ar[dr] & & \vdots \ar[dr] & & \vdots \ar[dr] & \\
    \cdots \ar[dr]& & {\scriptstyle (-7,6)} \ar[ur] \ar[dr] & & {\scriptstyle (-4,9)} \ar[ur] \ar[dr] & & {\scriptstyle (-1,12)} \ar[ur] \ar[dr] & & {\scriptstyle (2,15)} \ar[ur] \ar[dr] & & \cdots \\
    & {\scriptstyle (-7,3)} \ar[ur] \ar[dr] & & {\scriptstyle (-4,6)} \ar[ur] \ar[dr] & & {\scriptstyle (-1,9)} \ar[ur] \ar[dr] & & {\scriptstyle (2,12)} \ar[ur] \ar[dr] & & {\scriptstyle (5,15)} \ar[ur] \ar[dr] & \\
    \cdots \ar[ur]\ar[dr]& & {\scriptstyle (-4,3)} \ar[ur] \ar[dr] & & {\scriptstyle (-1,6)} \ar[ur] \ar[dr] & & {\scriptstyle (2,9)} \ar[ur] \ar[dr] & & {\scriptstyle (5,12)} \ar[ur] \ar[dr] & & \cdots\\
    & {\scriptstyle (-4,0)} \ar[ur] & & {\scriptstyle (-1,3)} \ar[ur] & & {\scriptstyle (2,6)} \ar[ur] & & {\scriptstyle (5,9)} \ar[ur] & & {\scriptstyle (8,12)} \ar[ur] & \\
               }
\]
\caption{The component $R$ of the Auslander-Reiten quiver  of $\C_3(A_{\infty})$.}
\label{figure1}
\end{figure}
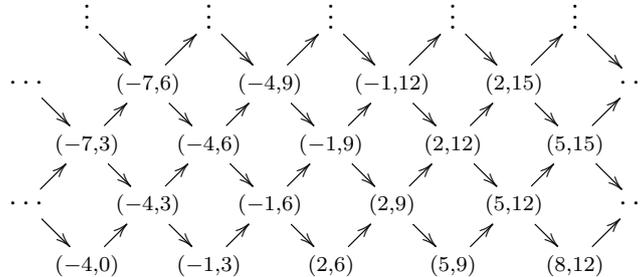
\begin{figure}[h]
\[
  \xymatrix @-3.0pc @! {
    & \vdots \ar[dr] & & \vdots \ar[dr] & & \vdots \ar[dr] & & \vdots \ar[dr] & & \vdots \ar[dr] & \\
    \cdots \ar[dr]& & {\scriptstyle (-8,5)} \ar[ur] \ar[dr] & & {\scriptstyle (-5,8)} \ar[ur] \ar[dr] & & {\scriptstyle (-2,11)} \ar[ur] \ar[dr] & & {\scriptstyle (1,14)} \ar[ur] \ar[dr] & & \cdots \\
    & {\scriptstyle (-8,2)} \ar[ur] \ar[dr] & & {\scriptstyle (-5,5)} \ar[ur] \ar[dr] & & {\scriptstyle (-2,8)} \ar[ur] \ar[dr] & & {\scriptstyle (1,11)} \ar[ur] \ar[dr] & & {\scriptstyle (4,14)} \ar[ur] \ar[dr] & \\
    \cdots \ar[ur]\ar[dr]& & {\scriptstyle (-5,2)} \ar[ur] \ar[dr] & & {\scriptstyle (-2,5)} \ar[ur] \ar[dr] & & {\scriptstyle (1,8)} \ar[ur] \ar[dr] & & {\scriptstyle (4,11)} \ar[ur] \ar[dr] & & \cdots\\
    & {\scriptstyle (-5,-1)} \ar[ur] & & {\scriptstyle (-2,2)} \ar[ur] & & {\scriptstyle (1,5)} \ar[ur] & & {\scriptstyle (4,8)} \ar[ur] & & {\scriptstyle (7,11)} \ar[ur] & \\
               }
\]

\caption{The component $\Sigma R$ of the Auslander-Reiten quiver  of $\C_3(A_{\infty})$.}
\label{figure2}
\end{figure}
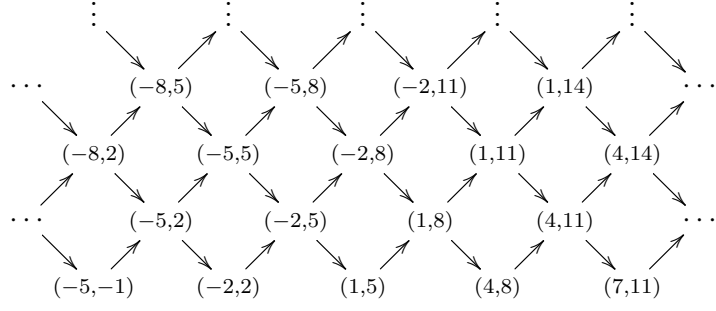
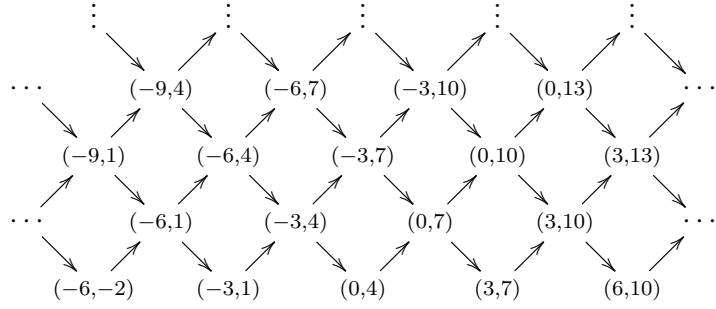
\begin{figure}[h]
\[
  \xymatrix @-3.0pc @! {
    & \vdots \ar[dr] & & \vdots \ar[dr] & & \vdots \ar[dr] & & \vdots \ar[dr] & & \vdots \ar[dr] & \\
    \cdots \ar[dr]& & {\scriptstyle (-9,4)} \ar[ur] \ar[dr] & & {\scriptstyle (-6,7)} \ar[ur] \ar[dr] & & {\scriptstyle (-3,10)} \ar[ur] \ar[dr] & & {\scriptstyle (0,13)} \ar[ur] \ar[dr] & & \cdots \\
    & {\scriptstyle (-9,1)} \ar[ur] \ar[dr] & & {\scriptstyle (-6,4)} \ar[ur] \ar[dr] & & {\scriptstyle (-3,7)} \ar[ur] \ar[dr] & & {\scriptstyle (0,10)} \ar[ur] \ar[dr] & & {\scriptstyle (3,13)} \ar[ur] \ar[dr] & \\
    \cdots \ar[ur]\ar[dr]& & {\scriptstyle (-6,1)} \ar[ur] \ar[dr] & & {\scriptstyle (-3,4)} \ar[ur] \ar[dr] & & {\scriptstyle (0,7)} \ar[ur] \ar[dr] & & {\scriptstyle (3,10)} \ar[ur] \ar[dr] & & \cdots\\
    & {\scriptstyle (-6,-2)} \ar[ur] & & {\scriptstyle (-3,1)} \ar[ur] & & {\scriptstyle (0,4)} \ar[ur] & & {\scriptstyle (3,7)} \ar[ur] & & {\scriptstyle (6,10)} \ar[ur] & \\
               }
\]

\caption{The component $\Sigma^2 R$ of the Auslander-Reiten quiver  of $\C_3(A_{\infty})$.}
\label{figure3}
\end{figure}
\end{example}
Now we recall the $\Hom$-space, or equivalently, the $\Ext$-space of two objects in $\C_n(A_{\infty})$ based on \cite{HJ1}.

\begin{definition}
Suppose that $x$ is an indecomposable object in $\ind( \C_n(A_{\infty})$.  Figure \ref{figure4} defines two
infinite sets $F^{\pm}( x )$ consisting of vertices in the same
component of the AR quiver as $x$.  Each set contains $x$ and all
other vertices inside the indicated boundaries, and the boundaries are
included in the sets.
\begin{figure}[h]
\[
\vcenter{
  \xymatrix @-3.0pc @! {
    &&&*{} &&&&&&&& \\
    &&&& *{} \ar@{--}[ul] & & & & *{} \ar@{--}[ur] \\
    &*{}&& F^-(x) & & & & & & F^+(x) && *{}\\
    &&*{}\ar@{--}[ul]& & & & {x} \ar@{-}[ddll] \ar@{-}[uull] \ar@{-}[ddrr] \ar@{-}[uurr]& & &&*{}\ar@{--}[ur]&\\
    && \\
    *{}\ar@{--}[r]&*{} \ar@{-}[rrr] && & *{}\ar@{-}[uull]\ar@{-}[rrrrrr]& & & & \ar@{-}[uurr]\ar@{-}[rrr]&&&*{}\ar@{--}[r]&*{}\\
           }
}
\]
\caption{The sets $F^{\pm}(x)$.}
\label{figure4}
\end{figure}
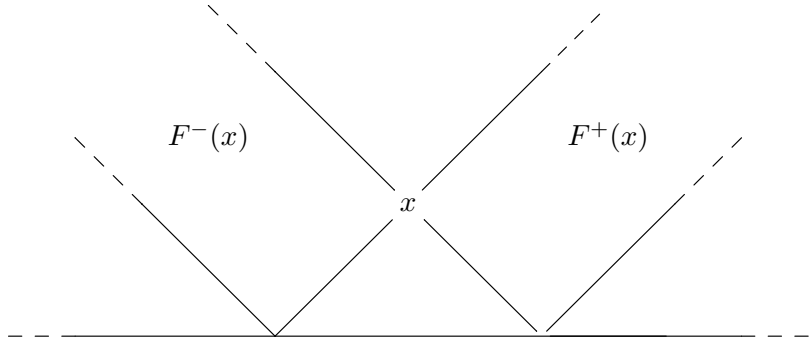
\end{definition}
\begin{lemma}\cite[Proposition 2.6]{HJ1}.
\label{cc}
Let $x, y \in \ind( \C_n(A_{\infty}) )$.  Then
\[
  \dim_k \C_n(A_{\infty})( x , y ) =
  \left\{
    \begin{array}{cl}
      1 & \mbox{for $y \in F^+( x ) \cup F^-( Sx )$}, \\[2pt]
      0 & \mbox{otherwise},
    \end{array}
  \right.
\]
where $S=\Sigma^{n+1}$ is the Serre functor of $\C_n(A_{\infty})$.
\end{lemma}
\subsection{$n$-cotorsion pairs in $\C_n(A_{\infty})$}
In this subsection, we give a geometric characterization of $n$-cotorsion pairs in $\C_n(A_{\infty})$. The following results are needed.
\begin{lemma}\cite[Proposition 3.4]{HJ1}
\label{aa}
Let $\X$ be a subcategory of $\C_n(A_{\infty})$ and let $\XX$ be the corresponding
set of $n$-admissible arcs.  The following conditions are equivalent.
\begin{enumerate}

  \item[\rm (1)]  The subcategory $\X$ is contravariantly finite.

\smallskip

  \item[\rm (2)]  Each right-fountain of $\XX$ is a left-fountain of $\XX$.

\end{enumerate}
\end{lemma}
Dually, the following result holds.
\begin{lemma}\cite[Proposition 3.5]{HJ1}
\label{bb}
Let $\X$ be a subcategory of $\C_n(A_{\infty})$ and let $\XX$ be the corresponding
set of $n$-admissible arcs.  The following conditions are equivalent.
\begin{enumerate}

  \item[\rm (1)]  The subcategory $\X$ is covariantly finite.

\smallskip

  \item[\rm (2)]  Each left-fountain of $\XX$ is a right-fountain of $\XX$.

\end{enumerate}
\end{lemma}

By Lemma \ref{cc} and Proposition 2.8 in \cite{HJ1}, we have the following result.

\begin{lemma}\label{g}
Suppose that $u, v$ are two $n$-admissible arcs, and $M_u, M_v$ are the corresponding indecomposable objects in $\C_n(A_{\infty})$. Then $u$ does not cross $v$ if and only if  $\Ext^{i}_{\C_n(A_{\infty})}(M_u,M_v)=0$ for all $1\leq i\leq n$.
\end{lemma}
Let  $\mathfrak{X}$ be a set of $n$-admissible arcs in $\Pi$. We define
$$\nc\mathfrak{X}=\{u\in\Pi\text{\;is\;an\;} n\text{-admissible\; arc\;}|\;u \text{\;does\;not\;cross\;any\;} n\text{-admissible\; arc\;in\;}\mathfrak{X}\}.$$
By Lemma \ref{g}, $\nc\mathfrak{X}$ corresponds to the subcategory $\bigcap\limits_{i=1}^{n}E(\mathfrak{X})[-i]^\perp=\bigcap\limits_{i=1}^{n}{^\bot}E(\mathfrak{X})[i]$.
\vspace{2mm}

We are ready to give the geometric characterization of $n$-cotorsion pairs in $\C_n(A_{\infty})$.
\begin{theorem}\label{dd}
Let $\X,\Y$ be subcategories of $\C_n(A_{\infty})$ and $\XX,\YY$ be the corresponding sets of $n$-admissible arcs. Then the following statements are equivalent.
\begin{enumerate}

  \item[\rm (1)]  $(\X,\Y)$ is an $n$-cotorsion pair.

\smallskip

  \item[\rm (2)] $\XX=\nc\YY$, $\YY=\nc\XX$,  each right-fountain of $\XX$ is a left-fountain of $\XX$, and each left-fountain of $\YY$ is a right-fountain of $\YY$.
\end{enumerate}
\end{theorem}
\begin{proof}
It is easy to check by Definition \ref{aaa}, Lemma \ref{aa}, Lemma \ref{bb}, and Lemma \ref{g}.
\end{proof}

We give a definition of Ptolemy diagram and obtain a necessary condition for an $n$-cotorsion pair.

\begin{definition}
Let $\mathfrak{X}$ be a set of $n$-admissible arcs in $\Pi$.  If for any two crossing $n$-admissible arcs $u=(r,s)$ and $v=(t,u)$ in $\mathfrak{X}$, those of $(r,t),(r,u),(s,t),(s,u)$, which are $n$-admissible arcs are in $\mathfrak{X}$, then $\mathfrak{X}$ is called a \emph{Ptolemy diagram}.
\end{definition}
\begin{lemma}\label{j}
Let $\X$ be a subcategory of $\C_n(A_{\infty})$, and $\XX$ be the corresponding set of $n$-admissible arcs. If $\XX=\nc\nc\XX$, then $\XX$ is a Ptolemy diagram.
\end{lemma}
\begin{proof}

Suppose that $u=(r,s)$ and $v=(t,u)$ are two crossing $n$-admissible arcs with $r<t<s<u$ in $\XX$, and suppose $(r,t)$ is an $n$-admissible arc, we need to show $(r,t)\in\XX=\nc\nc\XX$.

If $(r,t)\not\in\nc\nc\XX$, then there exists an $n$-admissible arc $w\in\nc\XX$ such that $w$ crosses $(r,t)$. Observe that every $n$-admissible arc $w\in\nc\XX$ crossing $(r,t)$ crosses either $u$ or $v$,  which is a contradiction with the assumption $u,v\in\XX$.
\end{proof}
By Theorem \ref{dd} and Lemma \ref{j}, we have the following corollary.
\begin{corollary}
Suppose that $\X$ is a subcategory of $\C_n(A_{\infty})$ and $\XX$ is the corresponding set of $n$-admissible arcs. If  $(\X,\Y)$ is an $n$-cotorsion pair for some subcategory $\Y$, then
\begin{enumerate}

  \item[\rm (1)] $\XX=\nc\nc\XX$ and each right-fountain of $\XX$ is a left-fountain of $\XX$.

\smallskip

  \item[\rm (2)]  $\XX$ is a Ptolemy diagram and each right-fountain of $\XX$ is a left-fountain of $\XX$.
\end{enumerate}
\end{corollary}
\begin{remark}
Note that when $n=1$, by Lemma 3.17 in \cite{Ng} and Corollary 5.9 in \cite{CZZ1}, we have that (1) and (2) are equivalent, and they both can imply $(\X,\Y)$ is an cotorsion pair for some subcategory $\Y$, but this is not always true for  $n\geq 2$.
\end{remark}
\begin{example}
Suppose $n=3$, the Auslander-Reiten quiver  of $\C_3(A_{\infty})$ is shown in Figures \ref{figure1}, \ref{figure2} and \ref{figure3}. Let $\X=\add((-4,0)\oplus(-4,3)\oplus(-1,3))$. One can check that $\X$ is a Ptolemy diagram, but it does not satisfy $\X=\nc\nc\X$.
\end{example}
\subsection{Geometric realization of mutation of $n$-cotorsion pairs in $\C_n(A_{\infty})$}
In this subsection, we give a geometric realization of mutation of $n$-cotorsion pairs in $\C_n(A_{\infty})$. First, we need the following definitions.
\begin{definition}
Let $(\X,\Y)$ be an $n$-cotorsion pair in $\C_n(A_{\infty})$ for some $\Y$ and $\XX$ be the corresponding set of $n$-admissible arcs. The set of $n$-admissible arcs in $\XX$ that does not cross any $n$-admissible arcs in $\XX$ is called the \emph{frame} of $\XX$, denoted by $F_{\XX}$.
\end{definition}
By the assumption above, the subcategory $E(F_{\XX})$ which corresponds to $F_{\XX}$ is the core of the $n$-cotorsion pair $(\X,\Y)$, i.e. $\I(\X)=E(F_{\XX})$.
\begin{definition}
\begin{itemize}
\item [(1)] For any $n$-admissible arc $(t,u)$ of $\Pi$, its \emph{rotation} in $\Pi$, denoted by $\rho_{\Pi}(t,u)$, is defined to be $(t-1,u-1)$, i.e.
      $$\rho_{\Pi}(t,u)=(t-1,u-1).$$
\item [(2)] Let $\mathfrak{D}$ be a set of non-crossing $n$-admissible arcs of $\Pi$.
\begin{itemize}
  \item [(i)] The $n$-admissible arcs in $\mathfrak{D}$ divide $\Pi$ into polygons, called $\mathfrak{D}$-\emph{cells}. Thus, any $n$-admissible arc $(t,u)$ which neither is in $\mathfrak{D}$ nor crosses any $n$-admissible arcs in $\mathfrak{D}$ is an $n$-admissible arc of a $\mathfrak{D}$-cell. We call the rotation of $(t,u)$ in this $\mathfrak{D}$-cell the $\mathfrak{D}$-\emph{rotation} of $(t,u)$, and denote it by $\rho_{\D}(t,u)$.
  \item [(ii)] Let $(\X,\Y)$ be an $n$-cotorsion pair in $\C_n(A_{\infty})$ for some $\Y$ and $\XX$ be the corresponding set of $n$-admissible arcs with $\mathfrak{D}\subset F_{\XX}$. The $\mathfrak{D}$-\emph{rotation} of $\XX$ is defined as
      $$\rho_{\mathfrak{D}}(\XX):=\{\rho_{\mathfrak{D}}(i,j)|(i,j)\in\XX\backslash\mathfrak{D}\}\cup\mathfrak{D}.$$
\end{itemize}
\end{itemize}
\end{definition}
\begin{example}
Suppose that $n=3$, the Auslander-Reiten quiver  of $\C_3(A_{\infty})$ is shown in Figures \ref{figure1}, \ref{figure2} and \ref{figure3}. Let $\X=\add\{(-4,3)\oplus(-4,6)\}$ and $\Y=\add\{(-4,3)\oplus(-4,6))\oplus(r,s)\oplus(t,u)\oplus(k,\ell)\;|\;r<s\leq-4, u>t\geq6, k<-4,\ell> 6\}$. One can check that $(\X,\Y)$ is a 3-cotorsion pair and $F_{\X}=\add\{(-4,3)\oplus(-4,6)\}$. Let $\mathfrak{D}=\add\{(-4,6)\}$. Then the $\mathfrak{D}$-rotation of $\X$ and $\Y$ are the followings, see Figures \ref{figure5} and \ref{figure6}.
$$\rho_{\mathfrak{D}}(\X)=\add\{(2,6)\oplus(-4,6)\}$$
$$\rho_{\mathfrak{D}}(\Y)=\add\{(2,6)\oplus(-4,6))\oplus(r,s)\oplus(t,u)\oplus(k,\ell)\;|\;r<s\leq-4, u>t\geq6, k<-4,\ell> 6\}.$$  For example, $\rho_{\mathfrak{D}}(-4,9)=(-5,8),\rho_{\mathfrak{D}}(-7,6)=(-8,-4).$
So $(\rho_{\mathfrak{D}}(\X),\rho_{\mathfrak{D}}(\Y))$ is still a 3-cotorsion pair.
\begin{figure}[ht]\centering
\begin{tikzpicture}
\draw[thick] (-6.5,0) to (3,0);
\draw[thick, dashed] (-7.5,0) to (-6.5,0) (4,0) to (3,0);
\draw[thick] (-6,0)node{$\bullet$}node[below]{-8}(-5.5,0)node{$\bullet$}node[below]{-7}(-5,0)node{$\bullet$}node[below]{-6}(-4.5,0)node{$\bullet$}node[below]{-5}(-4,0)node{$\bullet$}node[below]{-4} (-3.5,0)node{$\bullet$}node[below]{-3} (-3,0)node{$\bullet$}node[below]{-2} (-2.5,0)node{$\bullet$}node[below]{-1} (-2,0)node{$\bullet$}node[below]{0}
(-1.5,0)node{$\bullet$}node[below]{1} (-1,0)node{$\bullet$}node[below]{2} (-0.5,0)node{$\bullet$}node[below]{3} (0,0)node{$\bullet$}node[below]{4} (0.5,0)node{$\bullet$}node[below]{5}(1,0)node{$\bullet$}node[below]{6}(1.5,0)node{$\bullet$}node[below]{7}
(2,0)node{$\bullet$}node[below]{8}(2.5,0)node{$\bullet$}node[below]{9};
\draw[thick][red] (-4,0) ..  controls +(45:2) and +(145:2) .. (1,0);
\draw[thick] (-4,0) ..  controls +(45:1) and +(145:1) .. (-0.5,0);
\end{tikzpicture}

$\qquad\left\downarrow\rule{0cm}{1cm}\right.\{(-4,6)\}\text{-rotation}$

\begin{tikzpicture}
\draw[thick] (-6.5,0) to (3,0);
\draw[thick, dashed] (-7.5,0) to (-6.5,0) (4,0) to (3,0);
\draw[thick] (-6,0)node{$\bullet$}node[below]{-8}(-5.5,0)node{$\bullet$}node[below]{-7}(-5,0)node{$\bullet$}node[below]{-6}(-4.5,0)node{$\bullet$}node[below]{-5}(-4,0)node{$\bullet$}node[below]{-4} (-3.5,0)node{$\bullet$}node[below]{-3} (-3,0)node{$\bullet$}node[below]{-2} (-2.5,0)node{$\bullet$}node[below]{-1} (-2,0)node{$\bullet$}node[below]{0}
(-1.5,0)node{$\bullet$}node[below]{1} (-1,0)node{$\bullet$}node[below]{2} (-0.5,0)node{$\bullet$}node[below]{3} (0,0)node{$\bullet$}node[below]{4} (0.5,0)node{$\bullet$}node[below]{5}(1,0)node{$\bullet$}node[below]{6}(1.5,0)node{$\bullet$}node[below]{7}
(2,0)node{$\bullet$}node[below]{8}(2.5,0)node{$\bullet$}node[below]{9};
\draw[thick][red] (-4,0) ..  controls +(45:2) and +(145:2) .. (1,0);
\draw[thick] (-1,0) ..  controls +(45:1) and +(145:1) .. (1,0);
\end{tikzpicture}
\caption{A $\mathfrak{D}$-rotation of $\X$}
\label{figure5}
\end{figure}
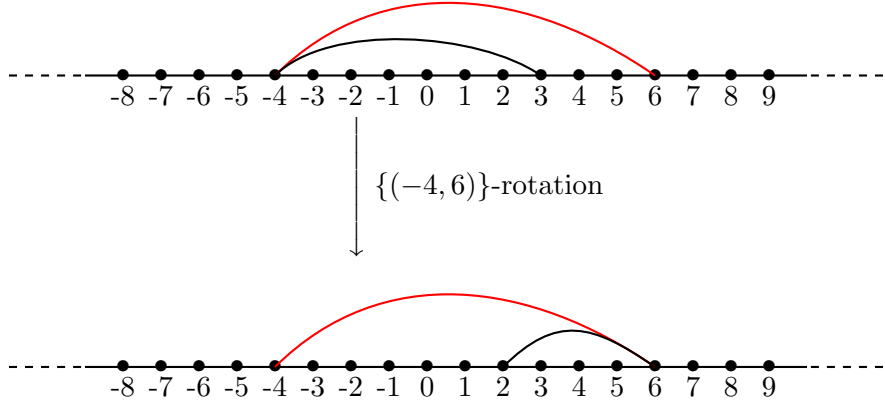
\begin{figure}[ht]\centering
\begin{tikzpicture}
\draw[thick] (-6.5,0) to (3,0);
\draw[thick, dashed] (-7.5,0) to (-6.5,0) (4,0) to (3,0);
\draw[thick] (-6,0)node{$\bullet$}node[below]{-8}(-5.5,0)node{$\bullet$}node[below]{-7}(-5,0)node{$\bullet$}node[below]{-6}(-4.5,0)node{$\bullet$}node[below]{-5}(-4,0)node{$\bullet$}node[below]{-4} (-3.5,0)node{$\bullet$}node[below]{-3} (-3,0)node{$\bullet$}node[below]{-2} (-2.5,0)node{$\bullet$}node[below]{-1} (-2,0)node{$\bullet$}node[below]{0}
(-1.5,0)node{$\bullet$}node[below]{1} (-1,0)node{$\bullet$}node[below]{2} (-0.5,0)node{$\bullet$}node[below]{3} (0,0)node{$\bullet$}node[below]{4} (0.5,0)node{$\bullet$}node[below]{5}(1,0)node{$\bullet$}node[below]{6}(1.5,0)node{$\bullet$}node[below]{7}
(2,0)node{$\bullet$}node[below]{8}(2.5,0)node{$\bullet$}node[below]{9};
\draw[thick][red] (-4,0) ..  controls +(45:2) and +(145:2) .. (1,0);
\draw[thick] (-4,0) ..  controls +(45:1) and +(145:1) .. (-0.5,0);
\draw[thick] (-4,0) ..  controls +(45:2.5) and +(145:2.5) .. (2.5,0);
\draw[thick] (-5.5,0) ..  controls +(45:2.5) and +(145:2.5) .. (1,0);
\draw[thick, dashed] (1,0) ..  controls +(45:1) and +(160:1) .. (4,1.5);
\draw[thick, dashed] (-4,0) ..  controls +(45:1) and +(160:1) .. (-6,1.5);
\end{tikzpicture}

$\qquad\left\downarrow\rule{0cm}{1cm}\right.\{(-4,6)\}\text{-rotation}$

\begin{tikzpicture}
\draw[thick] (-6.5,0) to (3,0);
\draw[thick, dashed] (-7.5,0) to (-6.5,0) (4,0) to (3,0);
\draw[thick] (-6,0)node{$\bullet$}node[below]{-8}(-5.5,0)node{$\bullet$}node[below]{-7}(-5,0)node{$\bullet$}node[below]{-6}(-4.5,0)node{$\bullet$}node[below]{-5}(-4,0)node{$\bullet$}node[below]{-4} (-3.5,0)node{$\bullet$}node[below]{-3} (-3,0)node{$\bullet$}node[below]{-2} (-2.5,0)node{$\bullet$}node[below]{-1} (-2,0)node{$\bullet$}node[below]{0}
(-1.5,0)node{$\bullet$}node[below]{1} (-1,0)node{$\bullet$}node[below]{2} (-0.5,0)node{$\bullet$}node[below]{3} (0,0)node{$\bullet$}node[below]{4} (0.5,0)node{$\bullet$}node[below]{5}(1,0)node{$\bullet$}node[below]{6}(1.5,0)node{$\bullet$}node[below]{7}
(2,0)node{$\bullet$}node[below]{8}(2.5,0)node{$\bullet$}node[below]{9};
\draw[thick][red] (-4,0) ..  controls +(45:2) and +(145:2) .. (1,0);
\draw[thick] (-1,0) ..  controls +(45:1) and +(145:1) .. (1,0);
\draw[thick] (-6,0) ..  controls +(45:1) and +(145:1) .. (-4,0);
\draw[thick] (-4.5,0) ..  controls +(45:2.5) and +(145:2.5) .. (2,0);
\draw[thick, dashed] (-4,0) ..  controls +(45:1) and +(160:1) .. (-2,1.5);
\draw[thick, dashed] (-4.5,0) ..  controls +(45:1) and +(160:1) .. (-6.5,1.5);
\end{tikzpicture}
\caption{A $\mathfrak{D}$-rotation of $\Y$}
\label{figure6}
\end{figure}
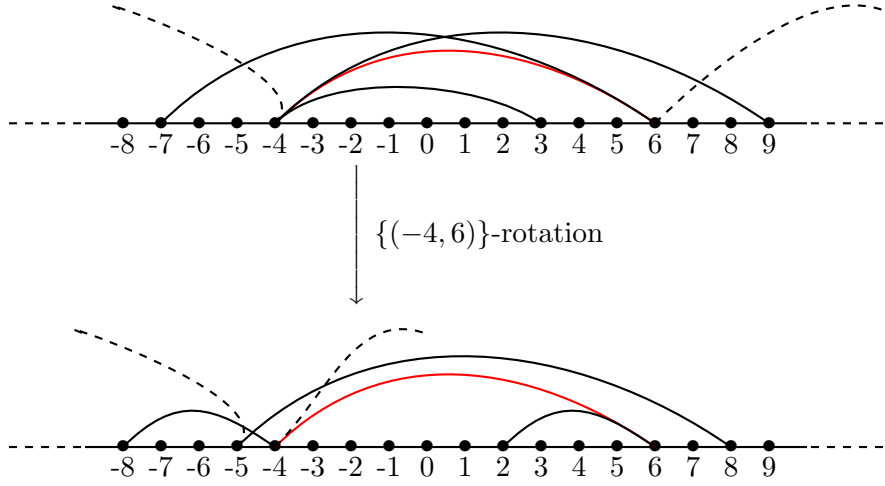
\end{example}

The following result is an easy observation by the Auslander-Reiten quiver  of $\C_n(A_{\infty})$.

\begin{lemma}\label{ee}
Suppose that $x=(r,s)$ and $y=(t,u)$ are two indecomposable objects in the same component of the Auslander-Reiten quiver  of $\C_n(A_{\infty})$ with $t<r<u<s$. Then there exists a triangle
$$(t,u)\rightarrow(t,s)\oplus(r,u)\rightarrow(r,s),$$
where $(x,y)$ is taken to be a zero object if $(x,y)$ is not an $n$-admissible arc.
\end{lemma}
Depending on the above observation, we have the following result.

\begin{lemma}\label{hh}
Suppose that $x=(r,s)$ and $y=(t,u)$ are two indecomposable objects in $\C_n(A_{\infty})$. Then
$\Ext^1(x,y)\neq 0$ if and only if $\Sigma^1 y\in F^{+}(x)\cup F^{-}(Sx)$. In this case, if $\Sigma^1 y\in F^{-}(Sx)$, then there exists a triangle
$$(t,u)\rightarrow(t,s)\oplus(r,u)\rightarrow(r,s),$$
if $\Sigma^1 y\in F^{+}(x)$, then there exists a triangle
$$(t,u)\rightarrow(s,u)\oplus(r,t)\rightarrow(r,s).$$
where $(x,y)$ is taken to be a zero object if $(x,y)$ is not an $n$-admissible arc.
\end{lemma}
\begin{proof}
By Lemma \ref{cc}, we have $\Ext^1(x,y)\neq 0$ if and only if $\Sigma^1 y\in F^{+}(x)\cup F^{-}(Sx)$. Now we show the existence of the triangle. First, suppose $\Sigma^1 y\in F^{-}(Sx)$. Then $x$ and $y$ are in the same component of the Auslander-Reiten quiver of $\C_n(A_{\infty})$. By the proof of Proposition 2.8 in \cite{HJ1}, we have $u\equiv s(\mod n)$, $t\leq r-n$ and $r+1\leq u\leq s-n$. Thus, $t<r<u<s$ and $(t,s)$ (resp. $(r,u)$) is either an $n$-admissible arc or an edge. So the triangle
$$(t,u)\rightarrow(t,s)\oplus(r,u)\rightarrow(r,s)$$
exists by Lemma \ref{ee}.

Suppose $\Sigma^1 y\in F^{+}(x)$. Then $x$ is in the ``next" component of $y$. Since $\Sigma^1 y\in F^{+}(x)$, by the proof of Proposition 2.8 in \cite{HJ1}, we have $u\equiv s+1(\mod n)$, $r+1\leq t \leq s-n$ and $s+1\leq u$. Thus, $r<t<s<u$ and $(r,t)$ (resp. $(s,u)$) is either an $n$-admissible arc or an edge. Observe that $(t,s+1)$ is an $n$-admissible arc (since $u\equiv s+1(\mod n)$ and $u\equiv t+1(\mod n)$, we have $s\equiv t(\mod n)$), and $(t,s+1)$ is in the same component as $(t,u)$. Moreover, $(s,u)$ is also in the same component as $(t,u)$, by Lemma \ref{ee}, there exists a  triangle
$$(t,s+1)\rightarrow(t,u)\rightarrow(s,u)\rightarrow(t-1,s),$$
where $(t-1,s)=\Sigma^1 (t,s+1)$ is in the same component as $x$. Note that $(r,t)$ and $x=(r,s)$ are in the same component, so are $(t-1,s)$ and $x$. By Lemma \ref{ee}, there exists a  triangle
$$(t,s+1)\rightarrow(r,t)\rightarrow(r,s)\rightarrow(t-1,s).$$
Since  $(s,u)$ and $(r,t)$ do not cross, $\Hom((s,u),\Sigma^1(r,t))=0$ by Lemma \ref{g}. We have a communicative diagram
$$\begin{array}{ccccccc}
(t,s+1)&\rightarrow&(t,u)&\rightarrow&(s,u)&\rightarrow&(t-1,s)\\
\parallel&&\downarrow &&\downarrow&&\parallel\\
(t,s+1)&\rightarrow&(r,t)&\rightarrow&(r,s)&\rightarrow&(t-1,s).
\end{array}$$
By Lemma 2.2 in \cite{XZ}, we get the triangle
$$(t,u)\rightarrow(s,u)\oplus(r,t)\rightarrow(r,s).$$
\end{proof}

\begin{theorem}\label{mainresult}
Let $(\X,\Y)$ be an $n$-cotorsion pair in $\C_n(A_{\infty})$ for some $\Y$, $\XX$ be the corresponding set of $n$-admissible arcs, and $F_{\XX}$ be the frame of $\XX$. Then the followings hold.
\begin{itemize}
  \item [\rm (1)] The frame $F_{\XX}$ of $\XX$ corresponds to the core of the $n$-cotorsion pair  $(\X,\Y)$, i.e. $\I(\X)=E(F_{\XX})$, where $E(F_{\XX})$ represents the subcategory of $\C_n(A_{\infty})$ which corresponds to $F_{\XX}$.
   \item [\rm (2)] Let $\mathfrak{D}\subset F_{\XX}$ be a set of non-crossing $n$-admissible arcs and $E(\mathfrak{D})$ be the corresponding subcategory of $\C_n(A_{\infty})$.  Then we have that
       $$\mu^{-1}_{E(\mathfrak{D})}(\X)=E(\rho_{\mathfrak{D}}(\XX)).$$
\end{itemize}
\end{theorem}
\begin{proof}
It is obvious that (1) holds.

In order to prove (2), it is enough to show that for any $n$-admissible arc $(t,u)$ in $\XX$ which is not in $\mathfrak{D}$, if $\rho_{\mathfrak{D}}(t,u)=(r,s)$, then $\mu^{-1}_{E(\mathfrak{D})}(t,u)=(r,s)$.

Let $\mathfrak{X}^\prime$ be the $\mathfrak{D}$-cell containing $(t,u)$. Since $\rho_{\mathfrak{D}}(t,u)=(r,s)$, $(r,s)$ is the rotation of $(t,u)$ in $\mathfrak{X}^\prime$. That means $(r,s)$ is an $n$-admissible arc of $\mathfrak{X}^\prime$, so $(r,s)$ does not cross any  $n$-admissible arc in $\mathfrak{D}$, we have $(r,s)\in\bigcap\limits_{i=1}^{n}{^{\perp_i}} E(\mathfrak{D})$.

Since $(r,s)$ is an $n$-admissible arc, $(t,u)$ crosses $(r,s)$. We have the following two cases.
\begin{enumerate}
  \item[(1)] Suppose $t<r<u<s$. Then $(r,u)$ and $(t, s)$ are edges of $\mathfrak{X}^\prime$. Thus, $(r,u)$ (resp. $(s,t)$) is  a zero object or a non-zero object in $E(\mathfrak{D})$, i.e. $(r,u)\oplus (s,t)\in E(\mathfrak{D})$. By Lemma \ref{hh}, there exists a triangle
   $$(t,u)\xrightarrow{f}(t,s)\oplus (r,u)\rightarrow (r,s).$$
   Since $(r,s)\in\bigcap\limits_{i=1}^{n}{^{\perp_i}}E(\mathfrak{D})$, this implies that $f$ is a left $E(\mathfrak{D})$-approximation and it is minimal since $(r,s)$ is indecomposable. So by definition, $\mu^{-1}_{E(\mathfrak{D})}(t,u)=(r,s)$ and we complete the proof of $\mu^{-1}_{E(\mathfrak{D})}(\X)=E(\rho_{\mathfrak{D}}(\XX)).$
  \item [(2)] Suppose $r<t<s<u$. Then $(r,t)$ and $(s, u)$ are edges of $\mathfrak{X}^\prime$. Thus, $(r,t)$ (resp. $(s,u)$) is  a zero object or a non-zero object in $E(\mathfrak{D})$, i.e. $(r,t)\oplus (s,u)\in E(\mathfrak{D})$. By Lemma \ref{hh}, there exists a triangle
   $$(t,u)\xrightarrow{f} (r,t)\oplus (s,u)\rightarrow (r,s).$$
   Since $(r,s)\in\bigcap\limits_{i=1}^{n}{^{\perp_i}}E(\mathfrak{D})$, this implies that $f$ is a left $E(\mathfrak{D})$-approximation and it is minimal since $(r,s)$ is indecomposable. So by definition, $\mu^{-1}_{E(\mathfrak{D})}(t,u)=(r,s)$ and we complete the proof of $\mu^{-1}_{E(\mathfrak{D})}(\X)=E(\rho_{\mathfrak{D}}(\XX)).$
\end{enumerate}
\end{proof}

\vspace{1cm}
\hspace{-5mm}\textbf{Data Availability}\hspace{2mm} Data sharing not applicable to this article as no datasets were generated or analysed during
the current study.
\vspace{2mm}

\hspace{-5mm}\textbf{Conflict of Interests}\hspace{2mm} The authors declare that they have no conflicts of interest to this work.


\begin{thebibliography}{99}
%
\bibitem{BM} K.~Baur, R.~Marsh. A geometric description of $m$-cluster categories. Trans. Amer. Math. Soc. 360(2008), 5789--5803.

\bibitem{BM1} K.~Baur, R.~Marsh. A geometric description of $m$-cluster categories of type $D_n$. Int. Math. Res. Not. 2007(2007), doi: 10.1093/imrn/rnm011.

\bibitem{CZZ} W. Chang, P. Zhou, B. Zhu. Cluster subalgebras and cotorsion pairs in Frobenius extriangulated
categories, 22(2019), 1051--1081.

\bibitem{CZZ1} H. Chang, Y. Zhou, B. Zhu. Cotorsion pairs in cluster categories of type $A^{\infty}_{\infty}$, 156(2018), 119--141.

\bibitem{CZ} H.~Chang, P.~Zhou. Mutation of $n$-cotorsion pairs in triangulated categories. J. Algebra, 667(2025), 653--671.

\bibitem{FHZZ} X. Fu, J. Hu, D. Zhang, H. Zhu. Balanced pairs on triangulated categories.
Algebra Colloq. 30(1)(2023), 385--394.

\bibitem{FZ} S.~Fomin, A.~Zelevinsky. Cluster algebras. I. Foundations. J. Amer. Math. Soc. 15(2)(2002), 497--529.

\bibitem{G} G. Jasso, n-abelian and n-exact categories, Math. Z. 283(3–4) (2016) 703–759.

\bibitem{H} M. Hovey.  Cotorsion pairs and model categories, in: Interactions Between Homotopy Theory and Algebra, in: Contemp. Math., vol.436, Amer. Math. Soc., Providence, RI, 2007, 277--296.
%
\bibitem{HZ} J. He, P. Zhou. On the relation between $n$-cotorsion pairs and $(n + 1)$-cluster tilting
subcategories. J. Algebra Appl. 21(2022), no. 1, 12 pages.
%
\bibitem{HJ} T.~Holm, P.~J{\o}rgensen. On a  cluster category of infinite Dynkin type, and the relation to triangulations of the infinity-gon. Math. Z. 270 (2012), 277--295.
    %
\bibitem{HJ1} T.~Holm, P.~J{\o}rgensen. Cluster tilting vs. weak cluster tilting in Dynkin type A infinity. Forum Math. 27 (2015), 1117--1137.
%
\bibitem{HJY} T. Holm, P.~J{\o}rgensen, D. Yang. Sparseness of $t$-structures and negative Calabi-Yau dimension in triangulated categories generated by a spherical object. Bull. Lond.
Math. Soc. 45(1)(2013),  120--130.


\bibitem{HMP} M. Huerta, O. Mendoza, M. A. P\'{e}rez. $n$-Cotorsion pairs. J. Pure Appl. Algebra 225(5) (2021) 106556.
%
\bibitem{HZZ} J. Hu, D. Zhang, P. Zhou. Proper classes and Gorensteinness in extriangulated categories, J. Algebra 551(2020), 23--60.
%
\bibitem{I} O. Iyama. Cluster tilting for higher Auslander algebras. Adv. Math. 226(2008): 1--61.
%
\bibitem{IY} O.~Iyama, Y.~Yoshino. Mutation in triangulated categories and rigid Cohen-Macaulay modules. Invent. Math. 172 (2008), no. 1, 117--168.

\bibitem{Ja} L.~Jacquet-Malo. A bijection between $m$-cluster tilting objects and $(m+2)$-angulations in $m$-cluster categories. J. Algebra. 595(2022), 581--632.

\bibitem{J} P. J{\o}rgensen. Auslander-Reiten theory over topological spaces. Comment. Math. Helv. 79 (2004),
160--182.



\bibitem{K} B. Keller. On triangulated orbit categories. Doc. Math. 10(2005), 551--581.

\bibitem{LN} Y. Liu, H. Nakaoka. Hearts of twin cotorsion pairs on extriangulated categories. J. Algebra
528(2019), 96--149.

\bibitem{Ng} P.~Ng. A characterization of torsion theories in the cluster category of type $A_\infty$. arXiv:1005.4364, 2010.
%
\bibitem{NP} H. Nakaoka, Y. Palu.  Extriangulated categories, Hovey twin cotorsion pairs and
model structures. Cah. Topol. G\'{e}om. Diff\'{e}r. Cat\'{e}g. 60(2)(2019), 117--193.


\bibitem{XZ} J. Xiao, B. Zhu.
Relations for the Grothendieck groups of triangulated categories. J. Algebra 257(2002), 37--50.


\bibitem{ZZ1} P. Zhou, B. Zhu. Triangulated quotient categories revisited, J. Algebra 502(2018), 196--232.

\bibitem{ZZ} Y. Zhou, B. Zhu. Mutation of torsion pairs in triangulated categories and its geometric realization.
Alg. and Rep. Theory 21(2018), 817--832.

\bibitem{ZZ20} B. Zhu, X. Zhuang. Tilting subcategories in extriangulated categories. Front. Math. China 15(2020), 225--253.

\end{thebibliography}
\end{document}